\documentclass[11pt]{article}
\usepackage{amssymb,amsmath,amsfonts,amsthm,mathrsfs}
\usepackage{graphicx,graphics,psfrag,epsfig,calrsfs,cancel}
\usepackage[colorlinks=true, pdfstartview=FitV, linkcolor=blue, citecolor=red, urlcolor=blue]{hyperref}

\usepackage{caption,subfig,color}
\usepackage{tikz}
\usetikzlibrary{calc}

\usepackage{geometry}
\newtheorem{assumption}{Assumption}
\newtheorem{remark}{Remark}
\newtheorem{lemma}{Lemma}
\newtheorem{proposition}{Proposition}
\newtheorem{theorem}{Theorem}
\newtheorem{corollary}{Corollary}
\newtheorem{definition}{Definition}

\newcommand{\N}{{\mathbb N}}
\newcommand{\Z}{{\mathbb Z}}
\newcommand{\R}{{\mathbb R}}
\newcommand{\C}{{\mathbb C}}

\newcommand{\A}{{\mathbb A}}
\newcommand{\B}{{\mathbb B}}
\newcommand{\E}{{\mathbb E}}
\newcommand{\F}{{\mathbb F}}
\newcommand{\M}{{\mathbb M}}
\newcommand{\D}{{\mathbb D}}
\newcommand{\U}{{\mathcal U}}
\newcommand{\Ubar}{\overline{\mathcal U}}
\newcommand{\cercle}{{\mathbb S}^1}
\newcommand{\eps}{\varepsilon}
\newcommand{\dps}{\displaystyle}

\begin{document}

\title{Fully discrete hyperbolic initial boundary value problems\\
with nonzero initial data}

\author{Jean-Fran\c{c}ois {\sc Coulombel}\thanks{CNRS and Universit\'e de Nantes, Laboratoire de Math\'ematiques 
Jean Leray (UMR CNRS 6629), 2 rue de la Houssini\`ere, BP 92208, 44322 Nantes Cedex 3, France. Email: 
{\tt jean-francois.coulombel@univ-nantes.fr}. Research of the author was supported by ANR project BoND, 
ANR-13-BS01-0009-01.}\\
$ $\\
{\it This article is dedicated to Denis Serre.}}
\date{\today}
\maketitle

\begin{abstract}
The stability theory for hyperbolic initial boundary value problems relies most of the time on the Laplace 
transform with respect to the time variable. For technical reasons, this usually restricts the validity of 
stability estimates to the case of zero initial data. In this article, we consider the class of {\it non-glancing} 
finite difference approximations to the hyperbolic operator. We show that the maximal stability estimates 
that are known for zero initial data and nonzero boundary source term extend to the case of nonzero 
initial data in $\ell^2$. The main novelty of our approach is to cover finite difference schemes with 
an arbitrary number of time levels. As an easy corollary of our main trace estimate, we recover former 
stability results in the semigroup sense by Kreiss \cite{kreiss1} and Osher \cite{osher1}.
\end{abstract}
\bigskip


\bigskip
\bigskip


Throughout this article, we use the notation
\begin{align*}
&\U := \{\zeta \in \C,|\zeta|>1 \}\, ,\quad \Ubar := \{\zeta \in \C,|\zeta| \ge 1 \}\, ,\\
&\D := \{\zeta \in \C,|\zeta|<1 \}\, ,\quad \cercle := \{\zeta \in \C,|\zeta|=1 \} \, .
\end{align*}
We let ${\mathcal M}_{d,p} ({\mathbb K})$ denote the set of $d \times p$ matrices with entries in ${\mathbb K} 
= \R \text{ or } \C$, and we use the notation ${\mathcal M}_d ({\mathbb K})$ when $p=d$. If $M \in {\mathcal M}_d 
(\C)$, $\text{sp} (M)$ denotes the spectrum of $M$, $\rho(M)$ denotes its spectral radius, and $M^*$ denotes the 
conjugate transpose of $M$. We let $I$ denote the identity matrix or the identity operator when it acts on an infinite 
dimensional space. We use the same notation $x^* \, y$ for the hermitian product of two vectors $x,y \in \C^d$ and 
for the euclidean product of two vectors $x,y \in \R^d$. The norm of a vector $x \in \C^d$ is $|x| := (x^* \, x)^{1/2}$. 
The corresponding norm on ${\mathcal M}_d (\C)$ is denoted $\| \cdot \|$. We let $\ell^2$ denote the set of square 
integrable sequences, without mentioning the indices of the sequences. Sequences may be valued in $\C^k$ for 
some integer $k$. In all this article, $N$ is a fixed positive integer.

\section{Introduction}
\label{intro}

We are interested in finite difference discretizations of hyperbolic initial boundary value problems. 
The {\it continuous} problem reads:
\begin{equation}
\label{ibvp}
\begin{cases}
\partial_t u + A \, \partial_x u = F(t,x) \, ,& (t,x) \in \R^+ \times \R^+ \, ,\\
B \, u(t,0) = g(t) \, ,& t \in \R^+ \, ,\\
u(0,x) = f(x) \, ,& x \in \R^+ \, ,
\end{cases}
\end{equation}
where, for simplicity, we consider the half-line $\R^+$ as the space domain. The matrix $A \in {\mathcal M}_N (\R)$ 
is assumed to be diagonalizable with real eigenvalues, and $B$ is a matrix - not necessarily a square one - that 
encodes the boundary conditions. The functions $F,g,f$ are given source terms, respectively, the interior source 
term, the boundary source term and the initial condition. Well-posedness for \eqref{ibvp} is equivalent to the 
{\it algebraic} condition:
\begin{equation*}
\text{\rm Ker } B \, \cap \, \text{\rm Span} \big( r_1,\dots,r_{N_+} \big) = \{ 0 \} \, ,
\end{equation*}
where the vectors $r_1,\dots,r_{N_+}$ span the unstable subspace of $A$, which corresponds to incoming characteristics. 
Furthermore, the matrix $B$ should have rank $N_+$. Provided these conditions are satisfied, the unique solution 
$u \in {\mathcal C}(\R^+;L^2(\R^+))$ to \eqref{ibvp} depends continuously on $f \in L^2(\R^+)$, $g \in L^2(\R^+)$ 
and $F \in L^2(\R^+ \times \R^+)$. We refer to \cite[chapter 4]{benzoni-serre} for a general presentation of the 
well-posedness theory for \eqref{ibvp}, as well as for its multidimensional analogue.

The well-posedness theory for finite difference discretizations of \eqref{ibvp} is far less clear. Let us first 
set a few notation. We let $\Delta x,\Delta t>0$ denote a space and a time step where the ratio $\lambda 
:=\Delta t/\Delta x$ is a fixed positive constant, and we also let $p,q,r,s$ denote some fixed integers. The 
solution to \eqref{ibvp} is approximated by means of a sequence $(U_j^n)$ defined for $n \in \N$, and 
$j\in 1-r+\N$. For $j=1-r,\dots,0$, the vector $U_j^n$ should be understood as an approximation of the 
trace $u(n\, \Delta t,0)$ on the boundary $\{ x=0 \}$. We consider finite difference approximations of 
\eqref{ibvp} that read:
\begin{equation}
\label{numibvp}
\begin{cases}
U_j^{n+1} = {\dps \sum_{\sigma=0}^s} Q_\sigma \, U_j^{n-\sigma} +\Delta t \, F_j^n \, ,& 
j\ge 1\, ,\quad n\ge s \, ,\\
U_j^{n+1} = {\dps \sum_{\sigma=-1}^s} B_{j,\sigma} \, U_1^{n-\sigma} +g_j^{n+1} \, ,& 
j=1-r,\dots,0\, ,\quad n\ge s \, ,\\
U_j^n = f_j^n \, ,& j\ge 1-r\, ,\quad n=0,\dots,s \, ,
\end{cases}
\end{equation}
where the operators $Q_\sigma$ and $B_{j,\sigma}$ are given by:
\begin{equation}
\label{defop}
Q_\sigma := \sum_{\ell=-r}^p A_{\ell,\sigma} \, {\bf T}^\ell \, ,\quad 
B_{j,\sigma} := \sum_{\ell=0}^q B_{\ell,j,\sigma} \, {\bf T}^\ell \, .
\end{equation}
In \eqref{defop}, all matrices $A_{\ell,\sigma},B_{\ell,j,\sigma}$ belong to ${\mathcal M}_N(\R)$ and 
are independent of the small parameter $\Delta t$, while ${\bf T}$ denotes the shift operator on the 
space grid: $({\bf T}^\ell v)_j :=v_{j+\ell}$.

Existence and uniqueness of a solution $(U_j^n)$ to \eqref{numibvp} is trivial since the numerical scheme 
is explicit, so the last requirement for well-posedness is continuous dependence of the solution on the three 
possible source terms $(F_j^n)$, $(g_j^n)$, $(f_j^n)$. This is a stability problem, and several definitions can 
be chosen. The following one dates back to the fundamental contribution \cite{gks}, and is specifically relevant 
when the boundary conditions are non-homogeneous ($(g_j^n) \not \equiv 0$):

\begin{definition}[Strong stability \cite{gks}]
\label{defstab1}
The finite difference approximation \eqref{numibvp} is said to be "strongly stable" if there exists a constant 
$C_1$ such that for all $\gamma>0$ and all $\Delta t \in \, ]0,1]$, the solution $(U_j^n)$ to \eqref{numibvp} 
with $(f_j^0) =\dots =(f_j^s) =0$ satisfies the estimate:
\begin{multline}
\label{stabilitenumibvp}
\dfrac{\gamma}{\gamma \, \Delta t+1} \, \sum_{n\ge s+1} \, \sum_{j\ge 1-r} 
\Delta t \, \Delta x \, {\rm e}^{-2\, \gamma \, n\, \Delta t} \, |U_j^n|^2 
+\sum_{n\ge s+1} \, \sum_{j=1-r}^p \Delta t \, {\rm e}^{-2\, \gamma \, n\, \Delta t} \, |U_j^n|^2 \\
\le C_1 \left\{ \dfrac{\gamma \, \Delta t+1}{\gamma} \, \sum_{n\ge s} \, \sum_{j\ge 1} 
\Delta t \, \Delta x \, {\rm e}^{-2\, \gamma \, (n+1)\, \Delta t} \, |F_j^n|^2 
+\sum_{n\ge s+1} \, \sum_{j=1-r}^0 \Delta t \, {\rm e}^{-2\, \gamma \, n\, \Delta t} \, |g_j^n|^2 
\right\} \, .
\end{multline}
\end{definition}

\noindent Another more common notion of stability only deals with nonzero initial data in \eqref{numibvp}, 
and was considered in the earlier publications \cite{kreiss1,osher1,osher2}:

\begin{definition}[Semigroup stability]
\label{defstab2}
The finite difference approximation \eqref{numibvp} is said to be "semigroup stable" if there exists a constant 
$C_2$ such that for all $\Delta t \in \, ]0,1]$, the solution $(U_j^n)$ to \eqref{numibvp} with $(F_j^n) =(g_j^n) 
=0$ satisfies the estimate:
\begin{equation}
\label{stabilitesemigroup}
\sup_{n\ge 0} \sum_{j \ge 1-r} \Delta x \, |U_j^n|^2 \le C_2 \, \sum_{n=0}^s \, \sum_{j \ge 1-r} \Delta x \, |f_j^n|^2 \, .
\end{equation}
\end{definition}

\begin{remark}
Both Definitions \ref{defstab1} and \ref{defstab2} are independent of the small parameter $\Delta t$ because 
of the fixed ratio $\Delta t/\Delta x$. We could therefore assume $\Delta t=1$, which we sometimes do later on, 
but have written \eqref{stabilitenumibvp} and \eqref{stabilitesemigroup} with $\Delta t$ and $\Delta x$ in order 
to make the connection with the "continuous" norms.
\end{remark}

Let us observe that semigroup stability for \eqref{numibvp} amounts to requiring that the (linear) operator
\begin{equation*}
(U^0,\dots,U^s) \longmapsto (U^1,\dots,U^{s+1}) \, ,
\end{equation*}
that is obtained by considering \eqref{numibvp} in the case $(F_j^n) =(g_j^n) =0$, be power bounded on $\ell^2 
\times \cdots \times \ell^2$. Let us quote \cite{TE} at this stage: \lq \lq The term {\it GKS-stable} is quite complicated. 
This is a special definition of stability (...) that involves exponential factors with respect to $t$ and other algebraic 
terms that remove it significantly from the more familiar stability notion of bounded norms of powers.\rq \rq 
$\,$ The goal of this article is to shed new light on the relations between these two notions of stability for 
\eqref{numibvp}.

There is clear evidence that semigroup stability does not imply strong stability for \eqref{numibvp}. One 
counter-example is given in \cite[page 361]{trefethen3}. In the PDE multidimensional context, a very simple 
counter-example can be constructed by considering the symmetric hyperbolic operator
\begin{equation*}
\partial_t +\begin{pmatrix}
1 & 0 \\
0 & -1 \end{pmatrix} \, \partial_{x_1} +\begin{pmatrix}
0 & 1 \\
1 & 0 \end{pmatrix} \, \partial_{x_2}
\end{equation*}
with maximally dissipative (but not strictly dissipative) boundary conditions. The maximal dissipativity 
property yields semigroup stability, see \cite[chapter 3]{benzoni-serre}, while the violation of the so-called 
Uniform Kreiss-Lopatinskii Condition precludes any trace estimate in $L^2$ of the solution in terms of the 
$L^2$ norm of the boundary source term.

Yet, a reasonable expectation is that strong stability does imply semigroup stability\footnote{This "uniform 
power boundedness conjecture" appears in an even stronger (!) version in \cite{kreiss-wu}.}. In the PDE 
multidimensional context, this was proved in \cite{kajitani,rauch} for both symmetric and strictly hyperbolic 
operators, later extended in \cite{audiard} to hyperbolic operators with constant multiplicity, and recently 
in \cite{metivier2} to an even wider class of hyperbolic operators. The symmetric case is more favorable 
and is easily dealt with by the introduction of auxiliary boundary conditions. Once again, the situation for 
difference approximations is not as complete. That strong stability implies semigroup stability is somehow 
hidden in the early works \cite{kreiss1,osher1,osher2} since the assumptions made there actually yield 
strong stability (even though only semigroup stability was proved then). The first general result on the 
"uniform power boundedness conjecture" dates back to \cite{wu} but is restricted to the case $s=0$ 
(numerical schemes with two time levels only) and to scalar problems. The analysis of \cite{wu} was 
generalized in \cite{jfcag} to the case of systems in any space dimension, still under the restriction $s=0$ 
and assuming that the discretized hyperbolic operator does not increase the $\ell^2$ norm on all $\Z$ 
($\Z^d$ in several space dimensions).

The present article is a first attempt to tackle the "uniform power boundedness conjecture" for schemes 
with more than two time levels, that is, when $s \ge 1$. Our main result, which is Theorem \ref{thm1} below, 
gives a trace estimate for the solution to \eqref{numibvp} in the case of nonzero initial data. We are not 
able yet to give a positive answer to the conjecture in a general framework, but we recover the results 
of \cite{kreiss1,osher1,osher2} as an easy corollary of Theorem \ref{thm1}. Unlike \cite{wu,jfcag}, our 
argument does not use the auxiliary Dirichlet boundary condition but relies on an easy summation by 
parts argument, as what one does for toy problems such as the upwind or Lax-Friedrichs schemes. 
Unfortunately, this summation by parts argument is restricted so far to the case $s=0$, but we do hope 
that our trace estimate for nonzero initial data does imply semigroup stability even for $s \ge 1$. This 
might require adapting the PDE arguments to the framework of difference approximations and is 
postponed to a future work.

\section{Assumptions and main result}
\label{section2}

We adopt the framework of \cite{jfc1,jfc2}. Let us first introduce the so-called amplification matrix:
\begin{equation}
\label{defA}
\forall \, \kappa \in \C \setminus \{ 0\} \, ,\, {\mathcal A}(\kappa) := \begin{pmatrix}
\widehat{Q_0}(\kappa) & \dots & \dots & \widehat{Q_s}(\kappa) \\
I & 0 & \dots & 0 \\
0 & \ddots & \ddots & \vdots \\
0 & 0 & I & 0 \end{pmatrix} \in {\mathcal M}_{N(s+1)}(\C)\, ,\quad 
\widehat{Q_\sigma}(\kappa) := \sum_{\ell=-r}^p \kappa^\ell \, A_{\ell,\sigma} \, .
\end{equation}
A necessary condition for both strong and semigroup stability of \eqref{numibvp} is that the discretization 
of the Cauchy problem be $\ell^2$ stable. We thus make our first assumption.

\begin{assumption}[Stability for the discrete Cauchy problem]
\label{assumption1}
There exists a constant $C>0$ such that the amplification matrix ${\mathcal A}$ in \eqref{defA} satisfies:
\begin{equation*}
\forall \, n \in \N \, ,\quad \forall \, \eta \in \R \, ,\quad 
\left\| {\mathcal A}({\rm e}^{i\, \eta})^n \right\| \le C \, .
\end{equation*}
In particular, the von Neumann condition $\rho({\mathcal A}({\rm e}^{i\, \eta})) \le 1$ holds.
\end{assumption}

\noindent We then make the following geometric regularity assumption on the difference operators $Q_\sigma$ 
in \eqref{numibvp}:

\begin{assumption}[Geometrically regular operator]
\label{assumption2}
The amplification matrix ${\mathcal A}$ defined by \eqref{defA} satisfies the following property: if 
$\underline{\kappa} \in \cercle$ and $\underline{z} \in \cercle \cap \text{\rm sp}({\mathcal A} 
(\underline{\kappa}))$ has algebraic multiplicity $\underline{\alpha}$, then there exist some 
functions $\zeta_1(\kappa),\dots,\zeta_{\underline{\alpha}}(\kappa)$ that are holomorphic in a 
neighborhood ${\mathcal W}$ of $\underline{\kappa}$ in $\C$, that satisfy
\begin{equation*}
\zeta_1(\underline{\kappa})=\dots=\zeta_{\underline{\alpha}}(\underline{\kappa})=\underline{z} \, ,\quad 
\det \big( z\, I -{\mathcal A}(\kappa) \big) = \vartheta(\kappa,z) \, 
\prod_{k=1}^{\underline{\alpha}} \big( z-\zeta_k(\kappa) \big) \, ,
\end{equation*}
with $\vartheta$ a holomorphic function of $(\kappa,z)$ in some neighborhood of 
$(\underline{\kappa},\underline{z})$ such that $\vartheta(\underline{\kappa},\underline{z}) \neq 0$, and 
furthermore, there exist some vectors $e_1(\kappa),\dots,e_{\underline{\alpha}}(\kappa) \in \C^{N(s+1)}$ 
that depend holomorphically on $\kappa \in {\mathcal W}$, that are linearly independent for all $\kappa 
\in {\mathcal W}$, and that satisfy
\begin{equation*}
\forall \, \kappa \in {\mathcal W} \, ,\quad \forall \, k=1,\dots,\underline{\alpha} \, ,\quad 
{\mathcal A}(\kappa) \, e_k(\kappa) =\zeta_k(\kappa) \, e_k(\kappa) \, .
\end{equation*}
\end{assumption}

\noindent Let us recall that in the scalar case ($N=1$), Assumption \ref{assumption2} is actually a consequence 
of Assumption \ref{assumption1}, see \cite[Lemma 7]{jfcnotes}. 
For technical reasons to be specified later in Section \ref{section3}, we make a final assumption on the amplification 
matrix ${\mathcal A}$:

\begin{assumption}[Non-glancing discretization]
\label{assumption3}
The amplification matrix ${\mathcal A}$ defined by \eqref{defA} satisfies the following property: if 
$\underline{\kappa} \in \cercle$ and $\underline{z} \in \cercle \cap \text{\rm sp}({\mathcal A} 
(\underline{\kappa}))$ has algebraic multiplicity $\underline{\alpha}$, then the eigenvalues 
$\zeta_1(\kappa),\dots,\zeta_{\underline{\alpha}}(\kappa)$ of ${\mathcal A}(\kappa)$ that 
are close to $\underline{z}$ when $\kappa$ is close to $\underline{\kappa}$ satisfy:
\begin{equation*}
\forall \, k=1,\dots,\underline{\alpha} \, ,\quad \zeta_k'(\underline{\kappa}) \neq 0 \, .
\end{equation*}
\end{assumption}

Many standard finite difference approximations satisfy Assumptions \ref{assumption1}, \ref{assumption2} 
and \ref{assumption3}, as for instance the upwind, Lax-Friedrichs and Lax-Wendroff schemes under 
a suitable CFL condition. The leap-frog approximation satisfies Assumptions \ref{assumption1} and 
\ref{assumption2} but violates Assumption \ref{assumption3}. The case $\zeta_k'(\underline{\kappa})=0$ 
gives rise to {\it glancing} wave packets with a vanishing group velocity, see \cite{trefethen1,trefethen3}. 
Here we assume that no such wave packet occurs.

For geometrically regular operators, the main results of \cite{jfc1,jfc2} show that strong stability 
is equivalent to an {\it algebraic condition}, known as the Uniform Kreiss-Lopatinskii Condition. Let 
us summarize the main steps in the analysis since some notation and results will be used later on. 
The main tool in the characterization of strong stability is the Laplace transform with respect to the 
time variable, which leads to the resolvent equation
\begin{equation}
\label{resolvent}
\begin{cases}
W_j -{\dps \sum_{\sigma=0}^s} z^{-\sigma-1} \, Q_\sigma \, W_j =F_j \, ,& j\ge 1\, ,\\
W_j -{\dps \sum_{\sigma=-1}^s} z^{-\sigma-1} \, B_{j,\sigma} \, W_1 =g_j \, ,& j=1-r,\dots,0\, ,
\end{cases}
\end{equation}
with $z \in \U$. The induction relation \eqref{resolvent} can be written in a more compact form by 
using an augmented vector. We introduce the matrices:
\begin{equation*}
\forall \, \ell=-r,\dots,p \, ,\quad \forall \, z \in \C \setminus \{ 0 \} \, ,\quad 
\A_\ell(z) := \delta_{\ell0}\, I-\sum_{\sigma=0}^s z^{-\sigma-1}\, A_{\ell,\sigma} \, ,
\end{equation*}
where $\delta_{\ell_1 \ell_2}$ denotes the Kronecker symbol. We also define the matrices
\begin{equation}
\label{defBlj}
\forall \, \ell =0,\dots,q \, ,\quad \forall \, j=1-r,\dots,0 \, ,\quad 
\forall \, z \in \C \setminus \{ 0 \} \, ,\quad 
\B_{\ell,j}(z) := \sum_{\sigma=-1}^s z^{-\sigma-1}\, B_{\ell,j,\sigma} \, .
\end{equation}
Our final assumption is rather standard and already appears in \cite{kreiss1}.

\begin{assumption}[Noncharacteristic discrete boundary]
\label{assumption4}
The matrices $\A_{-r}(z)$ and $\A_p(z)$ are invertible for all $z\in \Ubar$, or equivalently for 
all $z \in \C$ with $|z|>1-\eps_0$ for some $\eps_0 \in \, ]0,1]$.
\end{assumption}

Let us first consider the case $q<p$. In that case, all the $W_j$'s involved in the boundary conditions 
for the resolvent equation \eqref{resolvent} are coordinates of the augmented vector\footnote{Vectors 
are written indifferently in rows or columns in order to simplify the redaction.} ${\mathcal W}_1 := 
(W_p,\dots,W_{1-r}) \in \C^{N(p+r)}$. Using Assumption \ref{assumption4}, we can define a block 
companion matrix $\M(z)$ that is holomorphic on some open neighborhood ${\mathcal V} := \{ z \in \C 
\, , \, |z| >1-\eps_0 \}$ of $\Ubar$:
\begin{equation}
\label{defM}
\forall \, z \in {\mathcal V} \, ,\quad 
\M (z) := \begin{pmatrix}
-\A_p (z)^{-1} \, \A_{p-1}(z) & \dots & \dots & -\A_p (z)^{-1} \, \A_{-r}(z) \\
I & 0 & \dots & 0 \\
0 & \ddots & \ddots & \vdots \\
0 & 0 & I & 0 \end{pmatrix} \in {\mathcal M}_{N(p+r)}(\C) \, .
\end{equation}
We also define the matrix that encodes the boundary conditions for \eqref{resolvent}, namely
\begin{equation*}
\forall \, z \in \C \setminus \{ 0\} \, , \, \B (z):=\begin{pmatrix}
0 & \dots & 0 & -\B_{q,0}(z) & \dots & -\B_{0,0}(z) & I & & 0 \\
\vdots & & \vdots & \vdots & & \vdots & & \ddots & \\
0 & \dots & 0 & -\B_{q,1-r}(z) & \dots & -\B_{0,1-r}(z) & 0 & & I \end{pmatrix} \in 
{\mathcal M}_{Nr,N(p+r)}(\C) \, ,
\end{equation*}
with the $\B_{\ell,j}$'s defined in \eqref{defBlj}. With such definitions, it is a simple exercise to rewrite 
the resolvent equation \eqref{resolvent} as an induction relation for the augmented vector ${\mathcal W}_j 
:=(W_{j+p-1},\dots,W_{j-r}) \in \C^{N\, (p+r)}$, $j \ge 1$. This induction relation takes the form
\begin{equation}
\label{resolvent'}
\begin{cases}
{\mathcal W}_{j+1} = \M (z) \, {\mathcal W}_j +{\mathcal F}_j \, ,& j\ge 1\, ,\\
\B(z)\, {\mathcal W}_1 = {\mathcal G}\, ,& 
\end{cases}
\end{equation}
where the new source terms $({\mathcal F}_j),{\mathcal G}$ in \eqref{resolvent'} are given by:
\begin{equation*}
{\mathcal F}_j := (\A_p (z)^{-1} \, F_j,0,\dots,0) \, ,\quad {\mathcal G}:=(g_0,\dots,g_{1-r}) \, .
\end{equation*}
There is a similar equivalent form of \eqref{resolvent} in the case $q\ge p$, and we refer the reader 
to \cite[page 145]{jfcnotes} for the details. The main results of \cite{gks} and later \cite{jfc1,jfc2} 
characterize strong stability of \eqref{numibvp} in terms of an algebraic condition that involves 
the matrices $\M(z)$ and $\B(z)$ in \eqref{resolvent'}. This characterization of strong stability relies 
on a precise description of the stable and unstable spaces of the matrix $\M (z)$, including when 
$z$ becomes arbitrarily close to the unit circle. Some ingredients of the analysis are recalled and 
used in Section \ref{section3}.

Our main result is an estimate for the solution to \eqref{numibvp} with nonzero initial data. This estimate 
is entirely similar to \eqref{stabilitenumibvp} as far as the left hand-side of the inequality is concerned. 
Namely, we extend the known estimate for zero initial data to nonzero initial data by simply adding the 
$\ell^2$ norm of the initial data on the right hand-side of the inequality.

\begin{theorem}
\label{thm1}
Let Assumptions \ref{assumption1}, \ref{assumption2}, \ref{assumption3} and \ref{assumption4} be 
satisfied. If the scheme \eqref{numibvp} is strongly stable, then for all integer $P \in \N$, there exists 
a constant $C_P>0$ such that for all $\gamma>0$ and all $\Delta t \in \, ]0,1]$, the solution $(U_j^n)$ 
to \eqref{numibvp} with $(F_j^n) =(g_j^n) =0$ satisfies the estimate:
\begin{equation}
\label{estimthm}
\dfrac{\gamma}{\gamma \, \Delta t+1} \, \sum_{n\ge 0} \, \sum_{j\ge 1-r} 
\Delta t \, \Delta x \, {\rm e}^{-2\, \gamma \, n\, \Delta t} \, |U_j^n|^2 
+\sum_{n\ge 0} \, \sum_{j=1-r}^P \Delta t \, {\rm e}^{-2\, \gamma \, n\, \Delta t} \, |U_j^n|^2 
\le C_P \, \sum_{n=0}^s \, \sum_{j \ge 1-r} \Delta x \, |f_j^n|^2 \, .
\end{equation}
\end{theorem}

The analogue of the estimate \eqref{estimthm} is a key tool in \cite{kajitani} for proving the semigroup 
boundedness in the PDE multidimensional context. This requires however rather strong algebraic 
properties in order to justify some integration by parts argument (in a possibly non-symmetric context).

Let us now explain the links between Theorem \ref{thm1} and previous results in the literature, and 
more specifically with the analysis in \cite{osher1} (which is already an extension of \cite{kreiss1}). 
As explained earlier, Assumption \ref{assumption1} is necessary for any kind of stability result. It 
corresponds to condition (1) in the main Theorem of \cite{osher1} (see \cite[XIX]{osher1}). Assumption 
\ref{assumption2} is automatically satisfied in \cite{osher1} because the equations are scalar and the 
scheme involves only two time levels (recall that for $N=1$, Assumption \ref{assumption2} actually 
follows from Assumption \ref{assumption1}). Assumption \ref{assumption2} seems to be rather natural 
in one space dimension, whatever the values of $N$ and $s$, see the discussion in \cite[Section 2.2]{jfcnotes}. 
Assumption \ref{assumption3} is hidden in condition (2) of the main Theorem of \cite{osher1}, but allows 
for slightly more general situations. Eventually, strong stability corresponds to condition (4) in the main 
Theorem of \cite{osher1}. So at this stage, one might reasonably ask whether Theorem \ref{thm1} does 
imply the main result of \cite{osher1}, that is, semigroup stability of \eqref{numibvp} when $s=0$. This 
is the purpose of the following Corollary.

\begin{corollary}
\label{coro1}
In addition to Assumptions \ref{assumption1}, \ref{assumption2}, \ref{assumption3} and \ref{assumption4}, 
let us assume\footnote{All these extra assumptions are also present in \cite{osher1}.} $s=0$ and:
\begin{equation*}
\sum_{\ell=-r}^p A_{\ell,0} =I \, ,\quad \| Q_0 \|_{\ell^2(\Z) \rightarrow \ell^2(\Z)} =1 \, .
\end{equation*}
If the scheme \eqref{numibvp} is strongly stable, then it is also semigroup stable.
\end{corollary}

We emphasize that the decomposition technique used in \cite{osher1} does not seem to easily extend to 
the case $s \ge 1$, and this is the main reason why we advocate an alternative approach that is based on 
the trace estimate \eqref{estimthm} and a suitable integration by parts formula (see Section \ref{appendixA} 
for the proof of Corollary \ref{coro1}). Comparing with the derivation of semigroup estimates for \eqref{numibvp} 
in \cite{wu,jfcag}, the present approach is closer to the one that has been used in the PDE context, see e.g. 
\cite{kajitani,rauch,audiard}, and is also closer to the one that is used on toy problems such as the Lax-Friedrichs 
or upwind schemes, see \cite[chapter 11]{gko}.

The proof of Theorem \ref{thm1} is given in Section \ref{section3} and follows some arguments that appear in 
the surprisingly unnoticed\footnote{Actually, one of the main results of \cite{sarason1} shows that the uniform 
Lopatinskii condition is a sufficient condition for strong well-posedness of strictly hyperbolic initial boundary 
value problems, but the proof in \cite{sarason1} is restricted to constant coefficients linear systems, while the 
technique developed in \cite{kreiss2} extends to variable coefficients and therefore to nonlinear problems by 
fixed point iteration. Another main result in \cite{sarason1} gives stability estimates for solutions to initial 
boundary value problems with nonzero initial data, and this seems to be the first result of this kind for 
non-symmetric systems.} contribution \cite{sarason1}, see also \cite[section 5]{sarason2}. Our goal is to 
adapt such arguments to difference approximations and to make precise the new arguments involved in 
this extension. More precisely, the non-glancing Assumption is used in the proof of Theorem \ref{thm1} 
to show a trace estimate for the solution to the fully discrete Cauchy problem on $\Z$. Thanks to this trace 
estimate, we can incorporate the initial data for \eqref{numibvp} in the solution to a Cauchy problem, which 
reduces the study of \eqref{numibvp} to zero initial data and nonzero boundary source term. There is a wide 
literature on trace operators for hyperbolic Cauchy problems, see for instance the "well-known", though 
unpublished, reference \cite{MT} and works cited therein. We do not aim at a thorough description of the 
trace operator here, but rather focus on its $\ell^2$-boundedness. As explained in Appendix \ref{appA}, 
$\ell^2$-boundedness of the trace operator for the discrete Cauchy problem will be seen to be 
equivalent\footnote{The equivalent result for PDE problems seems to be part of folklore, though we have 
not found a detailed proof based on elementary arguments.} to the non-glancing condition in Assumption 
\ref{assumption3}.

\section{Proof of Theorem \ref{thm1}}
\label{section3}

From now on, we consider the scheme \eqref{numibvp} and assume that it is strongly stable in the 
sense of Definition \ref{defstab1}. When the interior and boundary source terms vanish, the scheme 
reads
\begin{equation}
\label{scheme}
\begin{cases}
U_j^{n+1} = {\dps \sum_{\sigma=0}^s} Q_\sigma \, U_j^{n-\sigma} \, ,& j\ge 1\, ,\quad n\ge s \, ,\\
U_j^{n+1} = {\dps \sum_{\sigma=-1}^s} B_{j,\sigma} \, U_1^{n-\sigma} \, ,& j=1-r,\dots,0\, ,\quad n\ge s \, ,\\
U_j^n = f_j^n \, ,& j\ge 1-r\, ,\quad n=0,\dots,s \, ,
\end{cases}
\end{equation}
with initial data $f^0,\dots,f^s \in \ell^2$.

All constants appearing in the estimates below are independent of the Laplace parameter $\gamma>0$, 
when present.

\subsection{Reduction to a Cauchy problem}

We decompose the solution $(U_j^n)$ to \eqref{scheme} as $U_j^n=V_j^n+W_j^n$, where $(V_j^n)$ satisfies 
a pure Cauchy problem that incorporates the initial data of \eqref{scheme}, and $(W_j^n)$ satisfies a system 
of the form \eqref{numibvp} with zero initial data and nonzero boundary source term. More precisely, $(V_j^n)$ 
denotes the solution to
\begin{equation}
\label{schemeaux1}
\begin{cases}
V_j^{n+1} = {\dps \sum_{\sigma=0}^s} Q_\sigma \, V_j^{n-\sigma} \, ,& j\in \Z \, ,\quad n\ge s \, ,\\
V_j^n = f_j^n \, ,& j \ge 1-r\, ,\quad n=0,\dots,s \, ,\\
V_j^n = 0 \, ,& j \le -r\, ,\quad n=0,\dots,s \, ,
\end{cases}
\end{equation}
and $(W_j^n)$ denotes the solution to
\begin{equation}
\label{schemeaux2}
\begin{cases}
W_j^{n+1} = {\dps \sum_{\sigma=0}^s} Q_\sigma \, W_j^{n-\sigma} \, ,& j\ge 1\, ,\quad n\ge s \, ,\\
W_j^{n+1} = {\dps \sum_{\sigma=-1}^s} B_{j,\sigma} \, W_1^{n-\sigma} +g_j^{n+1} \, ,& 
j=1-r,\dots,0\, ,\quad n\ge s \, ,\\
W_j^n = 0 \, ,& j\ge 1-r\, ,\quad n=0,\dots,s \, ,
\end{cases}
\end{equation}
where the source term $(g_j^n)$ in \eqref{schemeaux2} is defined by
\begin{equation}
\label{defgjn}
\forall \, j=1-r,\dots,0\, ,\quad \forall \, n\ge s+1 \, ,\quad 
g_j^n :=-V_j^n +\sum_{\sigma=-1}^s B_{j,\sigma} \, V_1^{n-1-\sigma} \, .
\end{equation}
The following result shows that Theorem \ref{thm1} only relies on a trace estimate for the solution to 
\eqref{schemeaux1}.

\begin{lemma}
\label{lem1}
Let Assumption \ref{assumption1} be satisfied. Assume furthermore that for all $P \in \N$, there exists a constant 
$C_P>0$, that does not depend on the initial data in \eqref{schemeaux1}, such that the solution $(V_j^n)$ to 
\eqref{schemeaux1} satisfies
\begin{equation}
\label{estimlem1}
\sum_{n \ge 0} \sum_{j=1-r}^P |V_j^n|^2 \le C_P \, \sum_{n=0}^s \, \sum_{j \ge 1-r} |f_j^n|^2 \, .
\end{equation}
Then the conclusion of Theorem \ref{thm1} holds.
\end{lemma}

\begin{proof}
Assumption \ref{assumption1} shows that the discrete Cauchy problem is stable in $\ell^2$, that is to say, 
there exists a numerical constant $C$ such that
\begin{equation*}
\sup_{n \ge 0} \, \sum_{j \in \Z} \Delta x \, |V_j^n|^2 \le C \, \sum_{n=0}^s \, \sum_{j \ge 1-r} \Delta x \, |f_j^n|^2 \, .
\end{equation*}
Introducing the parameter $\gamma >0$, and summing with respect to $n \in \N$, we get
\begin{equation*}
\dfrac{\gamma}{\gamma +1} \, \sum_{n \ge 0} \sum_{j \in \Z} \Delta x \, {\rm e}^{-2\, \gamma \, n} \, |V_j^n|^2 
\le C \, \dfrac{\gamma}{(1-{\rm e}^{-2\, \gamma}) \, (\gamma +1)} \, \sum_{n=0}^s \, \sum_{j \ge 1-r} \Delta x \, |f_j^n|^2 
\le C \, \sum_{n=0}^s \, \sum_{j \ge 1-r} \Delta x \, |f_j^n|^2 \, .
\end{equation*}
The substitution $\gamma \rightarrow \gamma \, \Delta t$ and the trace estimate \eqref{estimlem1} already 
yield:
\begin{equation}
\label{estimlem1Vn}
\dfrac{\gamma}{\gamma \, \Delta t+1} \, \sum_{n\ge 0} \, \sum_{j\ge 1-r} 
\Delta t \, \Delta x \, {\rm e}^{-2\, \gamma \, n\, \Delta t} \, |V_j^n|^2 
+\sum_{n\ge 0} \, \sum_{j=1-r}^P \Delta t \, {\rm e}^{-2\, \gamma \, n\, \Delta t} \, |V_j^n|^2 
\le C_P \, \sum_{n=0}^s \, \sum_{j \ge 1-r} \Delta x \, |f_j^n|^2 \, .
\end{equation}

The trace estimate \eqref{estimlem1} for $(V_j^n)$ gives a bound for the boundary source term $(g_j^n)$ in 
\eqref{defgjn}. Indeed, we have
\begin{equation*}
|g_j^n| \le C \, \sum_{\sigma=-1}^s \sum_{\ell=1-r}^{1+q} |V_\ell^{n-1-\sigma}| \, ,
\end{equation*}
with a constant $C$ that does not depend on $j$, $n$, nor on the sequence $(V_j^n)$. Introducing 
the parameter $\gamma>0$, we thus obtain
\begin{equation*}
\sum_{n \ge s+1} \sum_{j=1-r}^0 \Delta t \, {\rm e}^{-2\, \gamma \, n \, \Delta t} \, |g_j^n|^2 
\le C \, \sum_{n\ge 0} \, \sum_{j=1-r}^{1+q} \Delta t \, {\rm e}^{-2\, \gamma \, n\, \Delta t} \, |V_j^n|^2 
\le C \, \sum_{n=0}^s \, \sum_{j \ge 1-r} \Delta x \, |f_j^n|^2 \, ,
\end{equation*}
where we have used \eqref{estimlem1} again (with $P=1+q$). Since the scheme \eqref{numibvp} is 
strongly stable and \eqref{schemeaux2} starts with zero initial conditions, we can use the strong 
stability estimate and obtain
\begin{multline}
\label{estimlem1Wn}
\dfrac{\gamma}{\gamma \, \Delta t+1} \, \sum_{n\ge 0} \, \sum_{j\ge 1-r} 
\Delta t \, \Delta x \, {\rm e}^{-2\, \gamma \, n\, \Delta t} \, |W_j^n|^2 
+\sum_{n\ge 0} \, \sum_{j=1-r}^p \Delta t \, {\rm e}^{-2\, \gamma \, n\, \Delta t} \, |W_j^n|^2 \\
\le C \, \sum_{n \ge s+1} \sum_{j=1-r}^0 \Delta t \, {\rm e}^{-2\, \gamma \, n \, \Delta t} \, |g_j^n|^2 
\le C \, \sum_{n=0}^s \, \sum_{j \ge 1-r} \Delta x \, |f_j^n|^2 \, .
\end{multline}
The combination of both estimates \eqref{estimlem1Vn} and \eqref{estimlem1Wn} gives the conclusion 
of Theorem \ref{thm1}.
\end{proof}

Our goal now is to show that the trace estimate \eqref{estimlem1} is valid for the solution to the Cauchy 
problem \eqref{schemeaux1}. This is summarized in the following result.

\begin{proposition}
\label{prop1}
Let Assumptions \ref{assumption1}, \ref{assumption2}, \ref{assumption3} and \ref{assumption4} be 
satisfied. Then for all $P \in \N$, there exists a constant $C_P>0$ such that for all $\gamma>0$, the 
solution $(V_j^n)_{j \in \Z, n \in \N}$ to \eqref{schemeaux1} satisfies
\begin{equation*}
\sum_{n \ge 0} \sum_{j=1-r}^P {\rm e}^{-2\, \gamma \, n} \, |V_j^n|^2 \le C_P \, \sum_{n=0}^s \, 
\sum_{j \ge 1-r} |f_j^n|^2 \, .
\end{equation*}
\end{proposition}

\noindent Proposition \ref{prop1} clearly implies the validity of \eqref{estimlem1} by passing to the limit 
$\gamma \rightarrow 0$, and therefore the validity of Theorem \ref{thm1}. We thus now focus on the 
proof of Proposition \ref{prop1}, for which we first recall some fundamental properties of the matrix 
$\M(z)$ in \eqref{defM}.

\subsection{A brief reminder on the normal modes analysis}

The main result of \cite{jfc1} can be stated as follows.

\begin{theorem}[Block reduction of $\M$]
\label{thmjfc1}
Let Assumptions \ref{assumption1}, \ref{assumption2}, \ref{assumption3} and \ref{assumption4} be satisfied. 
Then for all $z \in \U$, the matrix $\M(z)$ in \eqref{defM} has $N\, r$ eigenvalues, counted with their multiplicity, 
in $\D \setminus \{ 0\}$, and $N\, p$ eigenvalues, counted with their multiplicity, in $\U$. We let $\E^s(z)$, resp. 
$\E^u(z)$, denote the $N\, r$-dimensional, resp. $N\, p$-dimensional, generalized eigenspace associated with 
those eigenvalues that lie in $\D \setminus \{ 0\}$, resp. $\U$.

Furthermore, for all $\underline{z} \in \Ubar$, there exists an open neighborhood ${\mathcal O}$ of $\underline{z}$ 
in $\C$, and there exists an invertible matrix $T(z)$ that is holomorphic with respect to $z \in {\mathcal O}$ such that:
\begin{equation*}
\forall \, z \in {\mathcal O} \, ,\quad T(z)^{-1}\, \M(z) \, T(z) =\begin{pmatrix}
\M_1(z) & & 0 \\
& \ddots & \\
0 & & \M_L(z) \end{pmatrix} \, ,
\end{equation*}
where the number $L$ of diagonal blocks and the size $\nu_\ell$ of each block $\M_\ell$ do not depend 
on $z\in {\mathcal O}$, and where each block satisfies one of the following three properties:
   \begin{itemize}
      \item  there exists $\delta>0$ such that for all $z \in {\mathcal O}$, $\M_\ell(z)^* \, \M_\ell(z) \ge 
             (1+\delta) \, I$,
      \item  there exists $\delta>0$ such that for all $z\in{\mathcal O}$, $\M_\ell(z)^* \, \M_\ell(z) \le 
             (1-\delta) \, I$,
      \item  $\nu_\ell=1$, $\underline{z}$ and $\M_\ell(\underline{z})$ belong to $\cercle$, and 
             $\underline{z} \, \M_\ell'(\underline{z}) \, \overline{\M_\ell(\underline{z})} \in \R 
             \setminus \{ 0 \}$.
   \end{itemize}
We refer to the blocks $\M_\ell$ as being of the first, second or third type.
\end{theorem}

Observe that Assumption \ref{assumption4} precludes the occurrence of blocks of the fourth type in the terminology 
of \cite{jfc1}, because such blocks only arise when glancing modes are present. In our framework, we shall only deal 
with elliptic blocks (first or second type) or scalar blocks. The latter correspond to eigenvalues of $\M(z)$ that depend 
holomorphically on $z$.

\subsection{Proof of the trace estimate for the Cauchy problem}

\subsubsection{The resolvent equation}

As already seen in the proof of Lemma \ref{lem1}, the solution $(V_j^n)$ to the Cauchy problem 
\eqref{schemeaux1} satisfies
\begin{equation}
\label{estim1}
\dfrac{\gamma}{\gamma +1} \, \sum_{n \ge 0} \sum_{j \in \Z} \Delta x \, {\rm e}^{-2\, \gamma \, n} \, |V_j^n|^2 
\le C \, \sum_{n=0}^s \, \sum_{j \ge 1-r} \Delta x \, |f_j^n|^2 \, ,
\end{equation}
for all $\gamma>0$. The estimate \eqref{estim1} shows that, for all $j \in \Z$, we can define the Laplace transform 
of the step function
\begin{equation*}
V_j(t) :=\begin{cases}
0 & \text{if } t<0 \, ,\\
V_j^n & \text{if } t\in [n,n+1[ \, , \, n \in \N \, .
\end{cases}
\end{equation*}
The Laplace transform $\widehat{V}_j$ is holomorphic in the right half-plane $\{ \text{\rm Re} \, \tau>0 \}$ 
for all $j \in \Z$, and Plancherel Theorem gives
\begin{equation*}
\forall \, \gamma >0 \, ,\quad \sum_{j \in \Z} \int_\R |\widehat{V}_j (\gamma+i\, \theta)|^2 \, {\rm d}\theta 
<+\infty \, .
\end{equation*}
In particular, for all $\gamma>0$, the sequence $\big( \widehat{V}_j (\gamma+i\, \theta) \big)_{j \in \Z}$ belongs 
to $\ell^2$ for almost every $\theta \in \R$.

Applying the Laplace transform to \eqref{schemeaux1} yields the resolvent equation on $\Z$:
\begin{equation}
\label{resolventZ1}
\forall \, j \in \Z \, ,\quad 
\widehat{V}_j(\tau) -{\dps \sum_{\sigma=0}^s} z^{-\sigma-1} \, Q_\sigma \, \widehat{V}_j(\tau) 
=F_j(\tau) \, ,
\end{equation}
where the source term $F_j$ is defined by
\begin{equation}
\label{defsourceFj}
\forall \, j \in \Z \, ,\quad 
F_j(\tau) :=\dfrac{1-z^{-1}}{\tau} \left\{ \sum_{n=0}^s z^{-n} \, f_j^n -\sum_{\ell=-r}^p \sum_{\sigma=0}^{s-1} 
\sum_{n=0}^{s-\sigma-1} z^{-n-\sigma-1} \, A_{\ell,\sigma} \, f_{j+\ell}^n \right\} \, ,
\end{equation}
and it is understood, as always in what follows, that $\tau$ is a complex number of positive real part $\gamma$, 
and $z:={\rm e}^\tau \in \U$. In \eqref{defsourceFj}, we use the convention $f_j^n=0$ if $j \le -r$. Using the matrix 
$\M(z)$ that has been defined in \eqref{defM}, we can rewrite \eqref{resolventZ1} as
\begin{equation}
\label{resolventZ2}
\forall \, j \in \Z \, , \, {\mathcal W}_{j+1}(\tau) = \M (z) \, {\mathcal W}_j (\tau) +{\mathcal F}_j(\tau)\, ,\quad 
{\mathcal W}_j (\tau) := \begin{pmatrix}
\widehat{V}_{j+p-1}(\tau) \\
\vdots \\
\widehat{V}_{j-r}(\tau) \end{pmatrix} , \, {\mathcal F}_j(\tau) :=\begin{pmatrix}
\A_p (z)^{-1} \, F_j(\tau) \\
 \\
0 \end{pmatrix} \, .
\end{equation}
Our goal now is to estimate the term ${\mathcal W}_{1-p-r}(\tau)$ of the solution $({\mathcal W}_j)$ to 
\eqref{resolventZ2}, and then to estimate finitely many ${\mathcal W}_\nu(\tau)$, $\nu \ge 1-p-r$.

\subsubsection{Estimates for $\gamma$ small}

In what follows, we always use the notation $\tau= \gamma +i\, \theta$, and we recall the notation $z:={\rm e}^\tau$. 
The source term ${\mathcal F}_j$ in \eqref{resolventZ2} is given in terms of $F_j$, whose expression is given 
in \eqref{defsourceFj}. The initial data $(f_j^0),\dots,(f_j^s)$ in \eqref{schemeaux1} vanish for $j \le -r$, and so 
therefore do $F_j$ and ${\mathcal F}_j$ for $j \le -p-r$ (and even for $j \le -r$ if $s=0$). This means that for all 
$j \le -p-r$, the sequence $({\mathcal W}_j)$ satisfies
\begin{equation*}
{\mathcal W}_{j+1}(\tau) = \M (z) \, {\mathcal W}_j (\tau) \, ,
\end{equation*}
and we know moreover that for all $\gamma>0$, the sequence $({\mathcal W}_j(\gamma+i\, \theta))_{j \in \Z}$ 
belongs to $\ell^2$ for almost every $\theta \in \R$. Applying Theorem \ref{thmjfc1}, this means that the vector 
${\mathcal W}_{1-p-r}(\tau)$ belongs to $\E^u(z)$ for almost every $\theta \in \R$.

Let us introduce the projectors $\Pi^s(z),\Pi^u(z)$ associated with the decomposition
\begin{equation*}
\C^{N\, (p+r)} =\E^s(z) \oplus \E^u(z) \, .
\end{equation*}
Using the exponential decay of $\M(z)^{-k} \, \Pi^u(z)$ as $k$ tends to infinity, the induction relation \eqref{resolventZ2} 
gives for almost every $\theta \in \R$:
\begin{equation}
\label{formuletrace}
{\mathcal W}_{1-p-r}(\tau) =\Pi^u(z) \, {\mathcal W}_{1-p-r}(\tau)
=-\sum_{j \ge 0} \M(z)^{-1-j} \, \Pi^u(z) \, {\mathcal F}_{1-p-r+j}(\tau) \, .
\end{equation}

We now focus on formula \eqref{formuletrace} and its consequences for small values of $\gamma$. More 
precisely, we consider a point $\underline{z}$ of the unit circle $\cercle$ and apply Theorem \ref{thmjfc1}. 
Let us introduce neighborhoods of the form as depicted in Figure \ref{fig1}:
\begin{equation*}
\forall \, \eps>0 \, ,\quad {\mathcal V}_{\underline{z},\eps} :=\Big\{ \underline{z} \, {\rm e}^{\alpha +i\, \beta} \, , \, 
\alpha, \beta \in \, ]-\eps,\eps[ \Big\} \, .
\end{equation*}
According to Theorem \ref{thmjfc1}, there exists some $\eps>0$ such that on ${\mathcal V}_{\underline{z},\eps}$, 
there is a holomorphic change of basis $T(z)$ that block-diagonalizes $\M(z)$, with blocks satisfying one of 
the properties stated in Theorem \ref{thmjfc1}. There is no loss of generality in assuming that blocks $\M_\ell(z)$ 
of the third type, which correspond to eigenvalues of $\M(z)$, can further be written as
\begin{equation}
\label{logMl}
\M_\ell(z) ={\rm e}^{\xi_\ell(z)} \, ,\quad \xi_\ell(\underline{z}) \in i \, \R \, ,\quad 
\underline{z} \, \xi_\ell'(\underline{z}) \in \R \setminus \{ 0\} \, ,
\end{equation}
where $\xi_\ell$ is holomorphic on ${\mathcal V}_{\underline{z},\eps}$ and 
\begin{equation*}
\forall \, z \in {\mathcal V}_{\underline{z},\eps} \, ,\quad |\text{\rm Re} \, (z\, \xi_\ell'(z))| \ge 
\dfrac{1}{2} \, |\xi_\ell'(\underline{z})|>0 \, .
\end{equation*}
In particular, $|\xi_\ell'|$ is uniformly bounded from below by a positive constant on ${\mathcal V}_{\underline{z},\eps}$. 
We can further assume that $T(z)$ and its inverse are uniformly bounded on ${\mathcal V}_{\underline{z},\eps}$.

\begin{remark}
Since $\M_\ell(z)$ is an eigenvalue of $\M(z)$, there holds $\xi_\ell(z) \not \in i \, \R$ for $z \in 
{\mathcal V}_{\underline{z},\eps} \cap \U$. More precisely, the $\xi_\ell(z)$'s of positive real part correspond 
to eigenvalues of $\M(z)$ in $\U$ (the unstable ones), and those of negative real part correspond to eigenvalues 
in $\D$ (the stable ones).
\end{remark}

Our goal is to derive a bound of the form
\begin{equation}
\label{estim2}
\int_{\mathcal I} |{\mathcal W}_{1-p-r}(\tau)|^2 \, {\rm d}\theta \le C \, \sum_{n=0}^s \, \sum_{j \ge 1-r} |f_j^n|^2 \, ,
\end{equation}
uniformly with respect to $\gamma \in \, ]0,\eps[$, where ${\mathcal I}$ denotes the set\footnote{Observe that 
the form of the neighborhood ${\mathcal V}_{\underline{z},\eps}$ implies that ${\mathcal I}$ is independent of 
$\gamma$, which is the reason for introducing such neighborhoods.}:
\begin{equation}
\label{formI1}
{\mathcal I}:= \{ \theta \in \R \, / \, {\rm e}^\tau \in {\mathcal V}_{\underline{z},\eps} \} 
=\cup_{k\in \Z} \, ]\underline{\theta}+2\, k \, \pi -\eps,\underline{\theta}+2\, k \, \pi +\eps[ \, ,\quad 
\underline{z}={\rm e}^{i\, \underline{\theta}} \, .
\end{equation}

\begin{figure}[t]
\begin{center}
\begin{tikzpicture}[scale=5]
\draw [domain=0.6:1.2,samples=30,smooth,dotted] plot ({(1/9)*cos(\x r)},{(1/9)*sin(\x r)});
\draw [domain=0.6:1.2,samples=30,smooth,dotted] plot ({(8/9)*cos(\x r)},{(8/9)*sin(\x r)});
\draw [domain=0.55:1.25,samples=30,smooth,thick] plot ({cos(\x r)},{sin(\x r)});
\draw [domain=0.6:1.2,samples=30,smooth,dotted] plot ({(11/9)*cos(\x r)},{(11/9)*sin(\x r)});
\draw[->] (0,0) -- (0,1);
\draw[->] (0,0) -- (1,0);
\draw [dotted] ({0*cos(0.6 r)},{0*sin(0.6 r)}) -- ({(11/9)*cos(0.6 r)},{(11/9)*sin(0.6 r)});
\draw [dotted] ({0*cos(1.2 r)},{0*sin(1.2 r)}) -- ({(11/9)*cos(1.2 r)},{(11/9)*sin(1.2 r)});
\draw ({cos(0.9 r)},{sin(0.9 r)}) circle (0.01cm);
\draw[<->,dotted] ({(8/9)*cos(1.25 r)},{(8/9)*sin(1.25 r)}) -- ({(11/9)*cos(1.25 r)},{(11/9)*sin(1.25 r)});
\node at ({0.18*cos(0.9 r)},{0.18*sin(0.9 r)}) {$2 \, \eps$};
\node at ({1.05*cos(1.42 r)},{1.05*sin(1.42 r)}) {$2 \, \sinh \eps$};
\node at ({cos(0.5 r)},{sin(0.5 r)}) {$\cercle$};
\node at ({cos(0.9 r)+0.07},{sin(0.9 r)}) {$\underline{z}$};
\node at ({1.1},0) {Re $z$};
\node at ({-0.1},{1}) {Im $z$};
\end{tikzpicture}
\end{center}
\caption{The neighborhood ${\mathcal V}_{\underline{z},\eps}$.}
\label{fig1}
\end{figure}

Let $\gamma \in \, ]0,\eps[$ be fixed. For almost every $\theta \in {\mathcal I}$, the vector ${\mathcal W}_{1-p-r} 
(\tau)$ is given by \eqref{formuletrace}, and we can diagonalize $\M(z)$ with the matrix $T(z)$. In order to 
cover all possible cases\footnote{If one type of block is not present in the reduction close to $\underline{z}$, 
the proof of \eqref{estim2} simplifies accordingly.}, we assume that the block diagonalization of $\M(z)$ reads
\begin{equation*}
T(z)^{-1} \, \M(z) \, T(z) =\text{\rm diag } (\M_\sharp(z),\M_\flat(z),\M_1^+(z),\dots,\M_{L^+}^+(z), 
\M_1^-(z),\dots,\M_{L^-}^-(z)) \, ,
\end{equation*}
where $\M_\sharp(z)$ is a block of the first type, $\M_\flat(z)$ is a block of the second type, and all other blocks are 
(scalars) of the third type with
\begin{align*}
&\forall \, \ell=1,\dots,L^+ \, ,\quad \M_\ell^+(\underline{z}) \in \cercle \, ,\quad \overline{\M_\ell^+(\underline{z})} 
\, \underline{z} \, (\M_\ell^+)'(\underline{z}) \in \R_+^* \, ,\\
&\forall \, \ell=1,\dots,L^- \, ,\quad \M_\ell^-(\underline{z}) \in \cercle \, ,\quad \overline{\M_\ell^-(\underline{z})} 
\, \underline{z} \, (\M_\ell^-)'(\underline{z}) \in \R_-^* \, .
\end{align*}
Then the generalized eigenspace $\E^u(z)$ is spanned by those column vectors of $T(z)$ which correspond 
to the blocks $\M_\sharp,\M_1^+,\dots,\M_{L^+}^+$, while the generalized eigenspace $\E^s(z)$ is spanned 
by those column vectors of $T(z)$ which correspond to the blocks $\M_\flat,\M_1^-,\dots,\M_{L^-}^-$, see, e.g., 
\cite[Lemma 3.3]{trefethen3} or \cite{jfc1}. An easy corollary of this "decoupling" property is that both projectors 
$\Pi^s,\Pi^u$ extend holomorphically to ${\mathcal V}_{\underline{z},\eps}$ and are bounded. We can even 
decompose $\Pi^u(z)$ as
\begin{equation*}
\Pi^u(z) =\Pi_\sharp (z)+\sum_{\ell=1}^{L^+} \Pi_\ell^+ (z) \, ,
\end{equation*}
with self-explanatory notation. For almost every $\theta \in {\mathcal I}$, the formula \eqref{formuletrace} 
then reads
\begin{align}
& \Pi_\sharp (z) \, {\mathcal W}_{1-p-r}(\tau)
=-\sum_{j \ge 0} \M(z)^{-1-j} \, \Pi_\sharp (z) \, {\mathcal F}_{1-p-r+j}(\tau) \, ,\label{formuletrace1}Ê\\
&\Pi_\ell^+ (z) \, {\mathcal W}_{1-p-r}(\tau) =-\sum_{j \ge 0} {\rm e}^{-(1+j) \, \xi_\ell^+(z)} \, \Pi_\ell^+ (z) \, 
{\mathcal F}_{1-p-r+j}(\tau) \, .\label{formuletrace2}
\end{align}

The norm of $\M(z)^{-1-j} \, \Pi_\sharp (z)$ decays exponentially with $j$, uniformly with respect to $z \in 
{\mathcal V}_{\underline{z},\eps}$, because $\M_\sharp$ is a block of the first type. Hence \eqref{formuletrace1} 
implies, with a constant $C$ that is uniform with respect to $\gamma$ and $\theta \in {\mathcal I}$:
\begin{equation*}
|\Pi_\sharp (z) \, {\mathcal W}_{1-p-r}(\tau)|^2 \le C \, \sum_{j \ge 0} |{\mathcal F}_{1-p-r+j}(\tau)|^2 \, .
\end{equation*}
We then use the definitions \eqref{resolventZ2} and \eqref{defsourceFj} to derive
\begin{equation*}
|\Pi_\sharp (z) \, {\mathcal W}_{1-p-r}(\tau)|^2 \le C \, \dfrac{|1-z^{-1}|^2}{|\tau|^2} \, 
\sum_{n=0}^s \, \sum_{j \ge 1-r} |f_j^n|^2 \, .
\end{equation*}
We end up with the estimate of the elliptic part of ${\mathcal W}_{1-p-r}$:
\begin{align}
\int_{\mathcal I} |\Pi_\sharp (z) \, {\mathcal W}_{1-p-r}(\tau)|^2 \, {\rm d}\theta &\le C \, \int_{\R}
\dfrac{|1-{\rm e}^{-\gamma-i\, \theta}|^2}{\gamma^2+\theta^2} \, {\rm d}\theta \, 
\sum_{n=0}^s \, \sum_{j \ge 1-r} |f_j^n|^2 \notag \\
&\le CÊ\, \dfrac{1-{\rm e}^{-2\, \gamma}}{\gamma} \, \sum_{n=0}^s \, \sum_{j \ge 1-r} |f_j^n|^2 
\le C \, \sum_{n=0}^s \, \sum_{j \ge 1-r} |f_j^n|^2 \, .\label{estim4}
\end{align}

We now turn to the hyperbolic components $\Pi_\ell^+ (z) \, {\mathcal W}_{1-p-r}(\tau)$, whose analysis 
relies on arguments that are similar to those in \cite{sarason1}. Since $\Pi_\ell^+ (z)$ projects on a 
one-dimensional vector space, we can rewrite \eqref{formuletrace2} as
\begin{equation*}
\Pi_\ell^+ (z) \, {\mathcal W}_{1-p-r}(\tau) =-\sum_{j \ge 0} {\rm e}^{-(1+j) \, \xi_\ell^+(z)} \, (L_\ell (z) \, 
{\mathcal F}_{1-p-r+j}(\tau)) \, T_\ell(z) \, ,
\end{equation*}
where $L_\ell$ is a row vector that depends holomorphically on $z$, and $T_\ell(z)$ is a column vector of $T(z)$. 
Using the expression \eqref{resolventZ2} of ${\mathcal F}_{1-p-r+j}(\tau)$, we find that, up to multiplying by harmless 
bounded functions of $z$, $\Pi_\ell^+ (z) \, {\mathcal W}_{1-p-r}(\tau)$ reads as a linear combination of the $s+1$ 
functions
\begin{equation*}
\dfrac{1-z^{-1}}{\tau} \, \sum_{j \ge 0} {\rm e}^{-j\, \xi_\ell^+(z)} \, f_{1-r+j}^n \, ,\quad n=0,\dots,s \, ,
\end{equation*}
which coincide, up to multiplying by harmless bounded functions of $z$, with:
\begin{equation*}
\dfrac{1-z^{-1}}{\tau} \, \F^n(\xi_\ell^+(z)) \, ,\quad n=0,\dots,s \, ,
\end{equation*}
where $\F^n$ denotes the Laplace transform of the initial condition
\begin{equation*}
f^n(x) :=\begin{cases}
f_{1-r+j}^n \, ,& x \in [j,j+1[ \, ,Ê\, j \in \N \, ,\\
0 \, ,& \text{\rm otherwise.}
\end{cases}
\end{equation*}
Recall that $\xi_\ell^+(z)$ has positive real part for $\gamma>0$, so the Laplace transform $\F^n$ is 
well-defined at $\xi_\ell^+(z)$. At this stage, the decomposition of $\Pi_\ell^+ (z) \, {\mathcal W}_{1-p-r}(\tau)$ 
implies the uniform bound
\begin{equation}
\label{estim5}
\int_{\mathcal I} |\Pi_\ell^+ (z) \, {\mathcal W}_{1-p-r}(\tau)|^2 \, {\rm d}\theta \le C \, \sum_{n=0}^s 
\int_{\mathcal I} \dfrac{|1-{\rm e}^{-\gamma-i\, \theta}|^2}{\gamma^2 +\theta^2} \, |\F^n(\xi_\ell^+(z))|^2 
\, {\rm d}\theta \, .
\end{equation}
We first simplify \eqref{estim5} by observing that $\theta$ enters the integrand on the right hand-side only 
through ${\rm e}^{i\, \theta}$ but at one place, which is the $1/(\gamma^2 +\theta^2)$ factor. The form 
\eqref{formI1} of ${\mathcal I}$ and some straightforward changes of variable turn \eqref{estim5} into
\begin{equation*}
\int_{\mathcal I} |\Pi_\ell^+ (z) \, {\mathcal W}_{1-p-r}(\tau)|^2 \, {\rm d}\theta \le C \, \sum_{n=0}^s 
\int_{\underline{\theta}-\eps}^{\underline{\theta}+\eps} |\F^n(\xi_\ell^+(z))|^2 \, {\rm d}\theta \, ,
\end{equation*}
with a constant $C$ that is still uniform with respect to $\gamma$. Because $|\xi_\ell'|$ is uniformly bounded 
away from zero on ${\mathcal V}_{\underline{z},\eps}$, we obtain
\begin{equation}
\label{estim6}
\int_{\mathcal I} |\Pi_\ell^+ (z) \, {\mathcal W}_{1-p-r}(\tau)|^2 \, {\rm d}\theta \le C \, \sum_{n=0}^s 
\int_{\underline{\theta}-\eps}^{\underline{\theta}+\eps} |\F^n(\xi_\ell^+(z))|^2 \, |i\, z \, (\xi_\ell^+)'(z)| \, {\rm d}\theta 
=C \, \sum_{n=0}^s \int_{{\mathcal C}_{\ell,\gamma}} |\F^n(z)|^2 \, |{\rm d}z| \, , 
\end{equation}
where ${\mathcal C}_{\ell,\gamma}$ denotes the (analytic) curve
\begin{equation}
\label{defClgamma}
{\mathcal C}_{\ell,\gamma} := \Big\{ \xi_\ell^+ \big( \underline{z} \, {\rm e}^{\gamma +i\, \theta} \big) \, ,Ê\, 
\theta \in \, ]-\eps,\eps[ \, \Big\} \, .
\end{equation}
The argument now relies on Carlson's Lemma \cite{carlson}, which gives a bound for curvilinear integrals 
of Laplace transforms in terms of the $L^2$ norm of the original function. More precisely, there holds
\begin{equation*}
\int_{{\mathcal C}_{\ell,\gamma}} |\F^n(z)|^2 \, |{\rm d}z| \le \dfrac{1}{\pi} \, 
\int_{i\, \R} |\F^n(w)|^2 \, A_\ell(\gamma,w) \, |{\rm d}w| \, , 
\end{equation*}
where $A_\ell(\gamma,w)$ denotes the total variation of the argument of $z-w$ as $z$ runs through the curve 
${\mathcal C}_{\ell,\gamma}$. In particular, if we can prove a uniform bound of the type
\begin{equation*}
\sup_{\gamma \in \, ]0,\eps[} \, \sup_{w \in i\, \R} A_\ell(\gamma,w) <+\infty \, ,
\end{equation*}
then we shall obtain from \eqref{estim6} and Carlson's Lemma the uniform bound
\begin{equation}
\label{estim7}
\int_{\mathcal I} |\Pi_\ell^+ (z) \, {\mathcal W}_{1-p-r}(\tau)|^2 \, {\rm d}\theta \le C \, \sum_{n=0}^s 
\sum_{j \ge 1-r} |f_j^n|^2 \, , 
\end{equation}
and the combination of \eqref{estim4} and \eqref{estim7} will yield \eqref{estim2}.

\subsubsection{Bounding the total variation of the argument}

The goal of this paragraph is to prove the following technical Lemma on families of analytic curves such as 
the ${\mathcal C}_{\ell,\gamma}$'s in \eqref{defClgamma}. We give a complete proof of this fact since the 
details in \cite{sarason1} are omitted and we consider an even more general situation than the corresponding 
one in \cite{sarason1}.

\begin{lemma}
\label{lemTVB}
Let $\eps>0$, and let $f$ be holomorphic on $]-\eps,\eps[^2 \subset \C$ with:
\begin{itemize}
 \item $f(0)=0$, $f'(0) \in \R^*_+$,
 
 \item for all $(\gamma,\theta) \in \, ]0,\eps[ \, \times \, ]-\eps,\eps[$, $f(\gamma+i\, \theta)$ has positive 
 real part.
\end{itemize}
For $w \in \R$ and $(\gamma,\theta) \in \, ]0,\eps[ \, \times \, ]-\eps,\eps[$, let $v(\gamma,\theta,w) \in \, 
]-\pi/2,\pi/2[$ denote the argument of $f(\gamma+i\, \theta)-i\, w$. Then, up to shrinking $\eps$, there 
exists a constant $C>0$ such that
\begin{equation}
\label{BVbound}
\sup_{\gamma \in \, ]0,\eps[} \, \sup_{w \in \R} \, \int_{-\eps}^\eps |\partial_\theta v(\gamma,\theta,w)| \, 
{\rm d}\theta \le C \, .
\end{equation}
\end{lemma}

\begin{proof}
There are two cases (see a similar argument in \cite[Proposition 4.5]{jfc2}). Since $f$ is holomorphic, then 
either $f(i\, \theta)$ is purely imaginary for all $\theta \in \, ]-\eps,\eps[$, or there exists a smallest $k \in \N^*$ 
and a constant $c>0$ such that
\begin{equation}
\label{dissipation}
\text{\rm Re } f(i\, \theta) \ge c \, \theta^{2\, k} \, .
\end{equation}
The proof of \eqref{BVbound} is different in each of these two cases. (The analysis in \cite{sarason1} only 
deals with the first case.)

Observe that we can always change $\eps$ for $\eps/2$, so that we can assume that $f$ together with any of 
its derivatives is bounded on the square $]-\eps,\eps[^2$.
\bigskip

$\bullet$ \underline{Case 1} : we assume that $f$ is such that $f(i\, \theta)$ is purely imaginary for all 
$\theta \in ]-\eps,\eps[$, which amounts to assuming $i^{n-1} \, f^{(n)}(0) \in \R$ for all $n \in \N$. Since 
$v$ denotes the argument of $f(\gamma+i\, \theta)-i\, w$, there holds
\begin{equation}
\label{lem2derivee}
\partial_\theta v (\gamma,\theta,w) =\text{\rm Re} \left( \dfrac{f'(\gamma+i\, \theta)}{f(\gamma+i\, \theta)-i\, w} 
\right) \, .
\end{equation}
The function $f$ is holomorphic and vanishes at $0$, so there exists a constant $C_0>0$, which does not 
depend on $\eps$, such that, up to choosing $\eps$ small enough, there holds
\begin{equation}
\label{bornef}
\sup_{(\gamma,\theta) \in \, ]-\eps,\eps[^2} |f(\gamma+i\, \theta)| \le C_0 \, \eps \, .
\end{equation}
The constant $C_0$ is now fixed. If $|w|>2\, C_0 \, \eps$ and $\gamma>0$, then \eqref{lem2derivee} yields
\begin{equation*}
|\partial_\theta v (\gamma,\theta,w)| \le \dfrac{1}{C_0 \, \eps} \, \sup_{(\gamma,\theta) \in \, ]-\eps,\eps[^2} 
|f'(\gamma+i\, \theta)| \, ,
\end{equation*}
and therefore
\begin{equation*}
\sup_{\gamma \in \, ]0,\eps[} \, \sup_{|w| \ge 2\, C_0 \, \eps} \, \int_{-\eps}^\eps |\partial_\theta v(\gamma,\theta,w)| 
\, {\rm d}\theta \le 2\, C_0 \, \sup_{(\gamma,\theta) \in \, ]-\eps,\eps[^2} |f'(\gamma+i\, \theta)| \, .
\end{equation*}
It therefore only remains to study the case $|w| \le 2\, C_0 \, \eps$, for which we are going to show that 
$\partial_\theta v$ is positive. The formula \eqref{lem2derivee} shows that $\partial_\theta v$ has the same 
sign as
\begin{equation*}
\text{\rm Re} \, \left( f'(\gamma+i\, \theta) \, \big( \overline{f(\gamma+i\, \theta)}+i\, w \big) \right) \, ,
\end{equation*}
and from the assumption on $f$, we find that $\partial_\theta v$ has the same sign as
\begin{equation*}
\text{\rm Re} \left( f'(\gamma+i\, \theta) \, \big( \overline{f(\gamma+i\, \theta)}+i\, w \big) 
-f'(i\, \theta) \, \big( \overline{f(i\, \theta)}+i\, w \big) \right) \, .
\end{equation*}
We rewrite the latter quantity as
\begin{equation*}
\text{\rm Re} \left( f'(i\, \theta) \, \big( \overline{f(\gamma+i\, \theta)}-\overline{f(i\, \theta)} \big) \right) 
-w\, \text{\rm Im} (f'(\gamma+i\, \theta)-f'(i\, \theta))
+\text{\rm Re} \left( (f'(\gamma+i\, \theta)-f'(i\, \theta)) \, \overline{f(\gamma+i\, \theta)} \right) \, ,
\end{equation*}
which, for $\eps$ sufficiently small, is bounded from below by (here we use $|w| \le 2\, C_0 \, \eps$):
\begin{equation*}
\dfrac{f'(0)^2}{2} \, \gamma -3 \, C_0 \, \eps \, \gamma \, 
\sup_{(\gamma,\theta) \in \, ]-\eps,\eps[^2} |f''(\gamma+i\, \theta)| \, .
\end{equation*}
In particular, for $\eps>0$ sufficiently small, there holds $\partial_\theta v(\gamma,\theta,w)>0$ for all 
$(\gamma,\theta) \in \, ]0,\eps[ \times ]-\eps,\eps[$ and $|w| \le 2\, C_0 \, \eps$. This property yields
\begin{equation*}
\sup_{\gamma \in \, ]0,\eps[} \, \sup_{|w| \le 2\, C_0 \, \eps} \, \int_{-\eps}^\eps |\partial_\theta v(\gamma,\theta,w)| 
\, {\rm d}\theta \le \pi \, ,
\end{equation*}
and \eqref{BVbound} holds.
\bigskip

$\bullet$ \underline{Case 2} : we now assume that $f$ satisfies \eqref{dissipation} for some minimal integer 
$k \in \N^*$, which amounts to assuming
\begin{equation*}
\forall \, n=0,\dots,2\, k-1 \, ,\quad i^{n-1} \, f^{(n)}(0) \in \R \, ,\quad 
\text{\rm and} \quad \text{\rm Re} \, ((-1)^k \, f^{(2\, k))}(0)) >0 \, .
\end{equation*}
We can still assume that $f$ satisfies \eqref{bornef} for some constant $C_0>0$, and therefore the same argument 
as in Case 1 gives a uniform bound for the total variation of $v$ when $|w| \ge 2\, C_0 \, \eps$, for in that case, $i\, w$ 
lies at a uniformly positive distance $C_0 \, \eps$ from the curves
\begin{equation*}
{\mathcal C}_\gamma := \big\{ f(\gamma +i\theta) \, ,Ê\, \theta \in \, ]-\eps,\eps[ \, \big\}\, .
\end{equation*}
Let us therefore consider from now on the case $|w| \le 2\, C_0 \, \eps$. We can assume that $f'$ does not vanish on 
$]-\eps,\eps[^2$, so the curve ${\mathcal C}_\gamma$ only consists of regular points. Hence its curvature equals, up 
to multiplying by a positive quantity:
\begin{equation*}
K(\gamma,\theta) := \text{\rm Re} \, f'(\gamma+i\, \theta) \, \text{\rm Im} \, f''(\gamma+i\, \theta) 
-\text{\rm Re} \, f''(\gamma+i\, \theta) \, \text{\rm Im} \, f'(\gamma+i\, \theta) 
=\text{\rm Re} \, \big( f'(\gamma+i\, \theta) \, \overline{f''(\gamma+i\, \theta)} \big) \, . 
\end{equation*}
Performing a Taylor expansion of $f'$ and $f''$, we compute
\begin{equation*}
K(0,\theta) =-\dfrac{f'(0) \, \text{\rm Re} \, ((-1)^k \, f^{(2\, k)}(0))}{(2\, k-2)!} \, \theta^{2\, k-2} +O(\theta^{2\, k-1}) \, .
\end{equation*}
Choosing $\eps$ small enough, this means that there exists positive constants $c$ and $C$, that do not 
depend on $\eps$, such that the curvature $K$ satisfies
\begin{equation*}
K(\gamma,\theta) \le -c\, \theta^{2\, k-2} +C \, \gamma \, .
\end{equation*}

If $k=1$, the curvature $K$ is uniformly negative, and we can conclude that the family of curves ${\mathcal C}_\gamma$, 
$0<\gamma<\eps$, consists of arcs of convex closed curves in the right half-place $\{ \text{\rm Re} \, \zeta >0 \}$. 
For $k=1$, this shows that the total variation
\begin{equation*}
\int_{-\eps}^\eps |\partial_\theta v(\gamma,\theta,w)| \, {\rm d}\theta \, ,
\end{equation*}
is not larger than $2\, \pi$ and the bound \eqref{BVbound} follows. In the case $k \ge 2$, we still have $K \le 0$ 
as long as $|\theta| \ge (\gamma/C)^{1/(2\, k-2)}$ for some suitable constant $C$, which means that the two arcs 
\begin{equation*}
\Big\{ f(\gamma +i\theta) \, ,Ê\, \theta \in \, ]-\eps,\max(-\eps,-(\gamma/C)^{1/(2\, k-2)})] \, \Big\} \, ,\quad 
\Big\{ f(\gamma +i\theta) \, ,Ê\, \theta \in \, [\min(\eps,(\gamma/C)^{1/(2\, k-2)}),\eps[ \, \Big\} \, ,
\end{equation*}
are convex\footnote{By convex, we mean that these curves are arcs of closed convex curves.}. In particular, 
there holds
\begin{equation}
\label{BVbound1}
\int_{]-\eps,\eps[ \setminus ]-(\gamma/C)^{1/(2\, k-2)},(\gamma/C)^{1/(2\, k-2)}[} 
|\partial_\theta v(\gamma,\theta,w)| \, {\rm d}\theta \le 4\, \pi \, .
\end{equation}

We now consider the regime where $\theta$ is small, meaning $|\theta| \le (\gamma/C)^{1/(2\, k-2)}$ with the same 
constant $C$ as the one for which \eqref{BVbound1} holds. We are going to show that in this regime, the derivative 
$\partial_\theta v$ is positive. Using \eqref{lem2derivee}, this derivative has the same sign as
\begin{equation*}
\text{\rm Re} \, \left( f'(\gamma+i\, \theta) \, \big( \overline{f(\gamma+i\, \theta)}+i\, w \big) \right) \, ,
\end{equation*}
which, similarly to what we did in Case 1, we rewrite as\footnote{The property $\text{\rm Re} \, (f'(i\, \theta) \, 
(\overline{f(i\, \theta)}+i\, w))=0$ does not hold any longer, and this is the reason why some new terms arise 
comparing to Case 1.}
\begin{multline*}
\text{\rm Re} \left( f'(i\, \theta) \, \big( \overline{f(\gamma+i\, \theta)}-\overline{f(i\, \theta)} \big) \right) 
-w\, \text{\rm Im} (f'(\gamma+i\, \theta)-f'(i\, \theta))
+\text{\rm Re} \left( (f'(\gamma+i\, \theta)-f'(i\, \theta)) \, \overline{f(\gamma+i\, \theta)} \right) \\
+\text{\rm Re} \left( f'(i\, \theta) \, \big( \overline{f(\gamma+i\, \theta)} +i\, w \big) \right) \, .
\end{multline*}
Using the same lower bounds as in Case 1, the latter quantity is lower bounded, for $\eps$ sufficiently 
small, by (here we use $|w| \le 2\, C_0 \, \eps$):
\begin{equation*}
\dfrac{f'(0)^2}{4} \, \gamma +\text{\rm Re} \left( f'(i\, \theta) \, \big( \overline{f(\gamma+i\, \theta)} +i\, w \big) \right) \, .
\end{equation*}
Performing a Taylor expansion for $f$ and $f'$, we have derived the following lower bound:
\begin{equation*}
|f(\gamma+i\, \theta)-i\, w|^2 \, \partial_\theta v \ge \dfrac{f'(0)^2}{4} \, \gamma -C \, \theta^{2\, k} 
-C\, \eps \, |\theta|^{2\, k-1} \ge c\, \gamma -C'\, \eps \, |\theta|^{2\, k-1} \, ,
\end{equation*}
for suitable constants $c,C'>0$. In the regime $\theta^{2\, k-2} \le \gamma/C$, $C$ fixed as in \eqref{BVbound1}, 
there holds $c\, \gamma -C'\, \eps \, |\theta|^{2\, k-1} \ge c\, \gamma/2$ for $\eps$ small enough, and we have 
thus shown that $\partial_\theta v$ is positive. This gives the bound
\begin{equation*}
\int_{]\min(-\eps,-(\gamma/C)^{1/(2\, k-2)}),\max(\eps,(\gamma/C)^{1/(2\, k-2)})[} 
|\partial_\theta v(\gamma,\theta,w)| \, {\rm d}\theta \le 2\, \pi \, ,
\end{equation*}
which we combine with \eqref{BVbound1} to derive
\begin{equation*}
\sup_{\gamma \in \, ]0,\eps[} \, \sup_{|w| \le 2\, C_0 \, \eps} \, 
\int_{-\eps}^\eps |\partial_\theta v(\gamma,\theta,w)| \, {\rm d}\theta \le 6\, \pi \, .
\end{equation*}
This completes the proof of \eqref{BVbound} in Case 2.
\end{proof}

The above proof of Lemma \ref{lemTVB} crucially uses the holomorphy of $f$, which corresponds, in the 
block reduction of $\M$, to the fact that there is no glancing frequency. When glancing frequencies occur, 
some eigenvalues of $\M$ display algebraic singularities, see \cite{gks}, that are combined with some 
possible dissipative behavior. A complete classification was made in \cite{jfc2}. The proof of the uniform 
BV bound \eqref{BVbound} is much more intricate when $f$ has an algebraic singularity at $0$, and we 
have not managed to complete it so far in a general framework.

Let us now explain how Lemma \ref{lemTVB} yields \eqref{estim7}. We consider a family of curves 
${\mathcal C}_{\ell,\gamma}$ in \eqref{defClgamma}. We can rewrite ${\mathcal C}_{\ell,\gamma}$ as
\begin{equation*}
{\mathcal C}_{\ell,\gamma} :=\underbrace{\xi_\ell^+ (\underline{z})}_{\in i\, \R} 
+\Big\{ f(\gamma +i\, \theta) \, ,Ê\, \theta \in \, ]-\eps,\eps[ \, \Big\} \, ,
\end{equation*}
with
\begin{equation*}
f(\gamma +i\, \theta) := \xi_\ell^+ \big( \underline{z} \, {\rm e}^{\gamma +i\, \theta} \big) -\xi_\ell^+ (\underline{z}) \, .
\end{equation*}
The function $f$ satisfies all the assumptions of Lemma \ref{lemTVB}, therefore, up to shrinking $\eps$, 
we can assume that the argument of $z-i\, w$, as $z$ runs through the curve ${\mathcal C}_{\ell,\gamma}$ 
and $w \in \R$, satisfies the uniform bound \eqref{BVbound}. Applying Carlson's Lemma, we have thus 
obtained \eqref{estim7}.

\subsubsection{Conclusion}

We still consider a fixed $\underline{z} \in \cercle$. Then for some sufficiently small $\eps>0$, we have shown 
that, uniformly with respect to the parameter $\gamma \in \, ]0,\eps[$, the estimate \eqref{estim2} holds.
For $\ell \ge 1-p-r$, we use the induction relation 
\eqref{resolventZ2}, and easily derive the uniform bound
\begin{equation*}
|{\mathcal W}_{\ell+1}(\tau)|^2 \le C_\ell \, |{\mathcal W}_{1-p-r}(\tau)|^2 +C_\ell \, \sum_{j=1-p-r}^\ell |F_j(\tau)|^2 \, .
\end{equation*}
Using \eqref{estim2} and the definition \eqref{defsourceFj}, we obtain
\begin{equation*}
\forall \, \ell \ge 1-p-r \, ,\quad 
\int_{\mathcal I} |{\mathcal W}_\ell(\tau)|^2 \, {\rm d}\theta \le C_\ell \, \sum_{n=0}^s \, \sum_{j \ge 1-r} |f_j^n|^2 \, ,
\end{equation*}
uniformly with respect to $\gamma \in \, ]0,\eps[$. We have therefore proved that for all $P \in \N$, there exists 
a constant $C_P>0$ such that
\begin{equation*}
\sum_{j=1-r}^P \int_{\mathcal I} |\widehat{V}_j(\tau)|^2 \, {\rm d}\theta \le C_P \, 
\sum_{n=0}^s \sum_{j \ge 1-r} |f_j^n|^2 \, .
\end{equation*}

We now use the compactness of $\cercle$ and cover it by finitely many neighborhoods ${\mathcal V}_{z_1,\eps_1}, 
\dots, {\mathcal V}_{z_K,\eps_K}$ such that, for each $k=1,\dots,K$ and $P \in \N$, there exists a constant $C_{k,P}$ 
for which there holds
\begin{equation}
\label{estim3}
\forall \, \gamma \in \, ]0,\eps_k[ \, ,\quad \sum_{j=1-r}^P 
\int_{{\mathcal I}_k} |\widehat{V}_j(\tau)|^2 \, {\rm d}\theta \le C_{k,P} \, \sum_{n=0}^s \sum_{j \ge 1-r} |f_j^n|^2 \, ,
\end{equation}
with the obvious notation
\begin{equation*}
{\mathcal I}_k := \{ \theta \in \R \, / \, {\rm e}^\tau \in {\mathcal V}_{z_k,\eps_k} \} \, .
\end{equation*}
The sets ${\mathcal I}_1,\dots,{\mathcal I}_K$ cover $\R$, so adding the estimates \eqref{estim3} gives
\begin{equation*}
\sum_{j=1-r}^P \int_\R |\widehat{V}_j(\tau)|^2 \, {\rm d}\theta \le C_P \, 
\sum_{n=0}^s \sum_{j \ge 1-r} |f_j^n|^2 \, ,
\end{equation*}
for $0<\gamma<\min \eps_k$, and some suitable constant $C_P>0$. Applying Plancherel Theorem, we get
\begin{equation*}
\sum_{n \in \N} \sum_{j=1-r}^P {\rm e}^{-2\, \gamma \, n} \, |V_j^n|^2 \le C_P \, 
\sum_{n=0}^s \sum_{j \ge 1-r} |f_j^n|^2 \, ,
\end{equation*}
for $\gamma \in \, ]0,\min \eps_k[$. This proves Proposition \ref{prop1} for sufficiently small values of $\gamma$.

The case where $\gamma$ is not close to zero\footnote{It should be understood that we use the scaling 
$\Delta t/\Delta x=\text{\rm Cst}$ and therefore only deal with one single parameter $\gamma$ but $\gamma$ 
is in fact a placeholder for $\gamma \, \Delta t$, so the regime "$\gamma$ small" can be thought of as that 
of the continuous limit $\Delta t \rightarrow 0$ with a fixed $\gamma$. It is then rather obvious that in that 
case, the trace estimate of Proposition \ref{prop1} can not be proved by just isolating the trace terms in 
\eqref{estim1}, which corresponds to passing from an $L^2_{t,x}$ estimate to an $L^2$ estimate at $x=0$.} 
is much easier for in that case, we already have the estimate \eqref{estim1}, and an obvious lower bound 
then gives
\begin{equation*}
\dfrac{\min \eps_k}{1+\min \eps_k} \, \sum_{n \in \N} \sum_{j=1-r}^P {\rm e}^{-2\, \gamma \, n} \, |V_j^n|^2 
\le C \, \sum_{n=0}^s \sum_{j \ge 1-r} |f_j^n|^2 \, ,
\end{equation*}
for $\gamma \ge \min \eps_k$. The proof of Proposition \ref{prop1}, and ultimately of Theorem \ref{thm1}, 
is thus complete.

\section{Proof of Corollary \ref{coro1}. The uniform power boundedness conjecture for schemes with two time levels}
\label{appendixA}

\subsection{The discrete Leibniz formula and integration by parts}

In this paragraph, we recall the discrete version of Leibniz formula and its consequence for integrating by 
parts. We recall that given $\nu \in \Z$ and a sequence $v=(v_j)_{j \ge \nu}$,  ${\bf T} v$ denotes the sequence 
defined by $({\bf T} v)_j :=v_{j+1}$ for all $j \ge \nu-1$, and ${\bf T}^{-1} v$ denotes the sequence defined by 
$({\bf T}^{-1} v)_j :=v_{j-1}$ for all $j \ge \nu+1$. (Of course, $v$ may also be indexed by all $\Z$.) Powers of 
${\bf T}$ and ${\bf T}^{-1}$ are defined similarly. We let ${\bf D}$ denote the operator ${\bf T}-I$, where $I$ 
is the identity. The operator ${\bf D}$ represents a discrete derivative\footnote{Our notation ${\bf D}$ corresponds 
to $D_+$ in \cite{gko}, but we omit the $+$ sign since we shall never use the other discrete derivative 
$D_-=I-{\bf T}^{-1}$, nor the centered derivative $D_0=({\bf T}-{\bf T}^{-1})/2$.}.

The following result is a discrete version of the Leibniz rule.

\begin{lemma}[Discrete Leibniz formula]
\label{lemLeibniz}
Let $u,v$ be two sequences with values in $\C^N$ and indexed either by $j \ge \nu$ for some $\nu \in \Z$, 
or by all $\Z$. Then for all $k \in \N$, there holds
\begin{equation*}
{\bf D}^k \, (u^* \, v) =\sum_{\underset{j_1 +j_2 \ge k}{j_1, j_2=0,}}^k 
\dfrac{k!}{(k-j_1)! \, (k-j_2)! \, (j_1+j_2-k)!} \, ({\bf D}^{j_1} u)^* \, {\bf D}^{j_2} v \, .
\end{equation*}
\end{lemma}

\begin{proof}
One starts with the formula
\begin{equation*}
{\bf D}^k \, (u^* \, v) =\sum_{j=0}^k \dfrac{k!}{(k-j)! \, j!} \, ({\bf D}^j u)^* \, {\bf T}^j {\bf D}^{k-j} v \, ,
\end{equation*}
which is obtained by a straightforward induction argument, and then use the binomial identity
\begin{equation*}
\forall \, j \in \N \, ,\quad {\bf T}^j =\sum_{\ell=0}^j \dfrac{j!}{(j-\ell)! \, \ell!} \, {\bf D}^\ell \, .
\end{equation*}
\end{proof}

\noindent The first consequence of Lemma \ref{lemLeibniz} is the following integration by parts formula, which 
mimics the analogous one for the product $u^* \, A \, u^{(k)}$, when $u$ is a $k$-times differentiable function 
and $A$ a hermitian matrix. Corollary \ref{coroipp1} below is a generalization of \cite[Lemma 11.1.1]{gko}.

\begin{corollary}
\label{coroipp1}
Let $A \in {\mathcal M}_N(\C)$ be hermitian and nonzero, and let $k \in \N^*$. Then there exists a unique 
hermitian form $q_{A,k}$ on $\C^{N\, k}$, and a unique collection of real numbers $\alpha_{1,k},\dots,\alpha_{k,k}$ 
that only depend on $k$ and not on $A$, such that for all sequence $u$ with values in $\C^N$, there holds
\begin{equation}
\label{ipp1}
2\, \text{\rm Re } \! (u^* \, A \, {\bf D}^k u) ={\bf D} \, \Big( q_{A,k} (u,\dots,{\bf D}^{k-1}u) \Big) 
+\sum_{j=1}^k \alpha_{j,k} \, ({\bf D}^j u)^* \, A \, {\bf D}^j u \, .
\end{equation}
\end{corollary}

\begin{proof}
Let us first prove the existence of the decomposition \eqref{ipp1}, which is done by induction. For $k=1$, 
one just uses Lemma \ref{lemLeibniz} and the fact that $A$ is hermitian to obtain
\begin{equation*}
2\, \text{\rm Re } \! (u^* \, A \, {\bf D} u) =\dfrac{1}{2} \, {\bf D} \, \Big( u^* \, A \, u \Big) 
-\dfrac{1}{2} \, ({\bf D} u)^* \, A \, {\bf D} u \, .
\end{equation*}
Let us therefore assume that the existence of the decomposition \eqref{ipp1} holds up to some integer $k$. 
We use Lemma \ref{lemLeibniz} and the fact that $A$ is hermitian to obtain
\begin{align*}
2\, \text{\rm Re } \! (u^* \, A \, {\bf D}^{k+1} u) ={\bf D} \, \Big( {\bf D}^k (u^* \, A \, u) \Big) 
&+\sum_{\underset{2\, j \ge k+1}{j=1,}}^{k+1} \star \, ({\bf D}^j u)^* \, A \, {\bf D}^j u \\
&+\sum_{\underset{j_1 +j_2 \ge k+1, j_1<j_2}{j_1, j_2=1,}}^{k+1} \star \, \, 
\text{\rm Re } \! (({\bf D}^{j_1} u)^* \, A \, {\bf D}^{j_2} u) \, ,
\end{align*}
where the $\star$ symbols represent harmless (real) numerical coefficients. Lemma \ref{lemLeibniz} shows 
that the term ${\bf D}^k (u^* \, A \, u)$ can be written as a hermitian form of $(u,\dots,{\bf D}^k u)$. The 
terms $\text{\rm Re } \! (({\bf D}^{j_1} u)^* \, A \, {\bf D}^{j_2} u)$, $j_1<j_2$, are simplified by using the 
induction assumption:
\begin{equation*}
\text{\rm Re } \! (({\bf D}^{j_1} u)^* \, A \, {\bf D}^{j_2} u) 
=\dfrac{1}{2} \, {\bf D} \, \Big( q_{A,j_2-j_1} ({\bf D}^{j_1} u,\dots,{\bf D}^{j_2-1}u) \Big) 
+\dfrac{1}{2} \, \sum_{j=1}^{j_2-j_1} \alpha_{j,j_2-j_1} \, ({\bf D}^{j_1+j} u)^* \, A \, {\bf D}^{j_1+j} u \, .
\end{equation*}
The existence of the decomposition \eqref{ipp1} up to the integer $k+1$ follows.

Let us now prove that the decomposition \eqref{ipp1} is unique. If two such decompositions exist, this means that 
we can find a hermitian form $q$ (which may depend on $A$), and a collection of real numbers $\alpha_{1,k}, \dots, 
\alpha_{k,k}$ (which do not depend on $A$) such that for all sequences $u$ with values in $\C^N$, there holds
\begin{equation}
\label{cor2eq1}
{\bf D} \, \Big( q(u,\dots,{\bf D}^{k-1}u) \Big) +\sum_{j=1}^k \alpha_{j,k} \, ({\bf D}^j u)^* \, A \, {\bf D}^j u =0 \, .
\end{equation}
Given arbitrary vectors $x_0,\dots,x_k \in \C^N$, we can find a sequence $u$ with values in $\C^N$ and indexed 
by $\N$, such that
\begin{equation*}
\forall \, j=0,\dots,kÊ\, ,\quad ({\bf D}^ju)_0 =x_j \, .
\end{equation*}
Equation \eqref{cor2eq1} evaluated at the index $\ell=0$ gives
\begin{equation*}
q(u_1,\dots,({\bf D}^{k-1}u)_1) -q(x_0,\dots,x_{k-1}) +\sum_{j=1}^k \alpha_{j,k} \, x_j^* \, A \, x_j =0 \, ,
\end{equation*}
that is to say
\begin{equation*}
q(x_1+x_0,\dots,x_k+x_{k-1}) -q(x_0,\dots,x_{k-1}) +\sum_{j=1}^k \alpha_{j,k} \, x_j^* \, A \, x_j =0 \, .
\end{equation*}
Let $Q$ denote the hermitian matrix associated with $q$, which therefore satisfies
\begin{equation}
\label{cor2eq2}
X_1^* \, Q \, X_1 +2\, \text{\rm Re } \! (X_0^* \, Q \, X_1) +\sum_{j=1}^k \alpha_{j,k} \, x_j^* \, A \, x_j =0 \, ,\quad 
X_0 :=\begin{pmatrix}
x_0 \\
\vdots \\
x_{k-1} \end{pmatrix} ,Ê\, X_1 :=\begin{pmatrix}
x_1 \\
\vdots \\
x_k \end{pmatrix} \, .
\end{equation}
The vector $x_0$ enters Equation \eqref{cor2eq2} only through the term $\text{\rm Re } \! (X_0^* \, Q \, X_1)$. 
Since $x_0,\dots,x_k$ are arbitrary in $\C^N$, the block decomposition of $Q$:
\begin{equation*}
Q =\begin{pmatrix}
Q_{0,0} & \dots & Q_{0,k-1} \\
\vdots &  & \vdots \\
Q_{k-1,0} & \dots & Q_{k-1,k-1} \end{pmatrix} \, ,
\end{equation*}
necessarily satisfies $Q_{0,0} =\cdots =Q_{0,k-1} =0$. Since $Q$ is hermitian, this implies of course $Q_{0,0} =\cdots 
=Q_{k-1,0} =0$. In other words, the hermitian form $q$ only depends on its $k-1$ last arguments, which reduces 
\eqref{cor2eq2}, with obvious notation, to
\begin{equation*}
Y_1^* \, \widetilde{Q} \, Y_1 +2\, \text{\rm Re } \! (Y_0^* \, \widetilde{Q} \, Y_1) 
+\sum_{j=1}^k \alpha_{j,k} \, x_j^* \, A \, x_j =0 \, ,\quad Y_0 :=\begin{pmatrix}
x_1 \\
\vdots \\
x_{k-1} \end{pmatrix} ,Ê\, Y_1 :=\begin{pmatrix}
x_2 \\
\vdots \\
x_k \end{pmatrix} \, .
\end{equation*}
Choosing $Y_1=0$ in the latter relation gives $\alpha_{1,k} \, x_1^* \, A \, x_1=0$, which means that $\alpha_{1,k}$ 
equals zero (here we use the fact that $A$ is nonzero), and uniqueness of the decomposition \eqref{ipp1} follows 
by induction on $k$.
\end{proof}

\noindent Let us observe that in Corollary \ref{coroipp1}, if $A$ is a real symmetric matrix, then the corresponding 
$q_{A,k}$ is a real quadratic form on $\R^{N\, k}$. The proof of Corollary \ref{coro1} requires an extension of 
Corollary \ref{coroipp1} to the case where $A$ is real and skew-symmetric, which we state now.

\begin{corollary}
\label{coroipp2}
Let $A \in {\mathcal M}_N(\R)$ be skew-symmetric and nonzero, and let $k \in \N$, $k \ge 2$. Then there exists a unique 
quadratic form $q_{A,k}$ on $\R^{N\, k}$, and a unique collection of real numbers $\beta_{1,k},\dots,\beta_{k-1,k}$ that 
only depend on $k$ and not on $A$, such that for all sequence $u$ with values in $\R^N$, there holds
\begin{equation}
\label{ipp2}
u^* \, A \, {\bf D}^k u ={\bf D} \, \Big( q_{A,k} (u,\dots,{\bf D}^{k-1}u) \Big) +\sum_{j=1}^{k-1} \beta_{j,k} \, ({\bf D}^j u)^* 
\, A \, {\bf D}^{j+1} u \, .
\end{equation}
\end{corollary}

\begin{proof}
The proof follows closely that of Corollary \ref{coroipp2}. We briefly indicate the induction argument for the 
existence of the decomposition \eqref{ipp2}. For $k=2$, we use Lemma \ref{lemLeibniz} and the fact that 
$A$ is skew-symmetric to obtain
\begin{equation*}
u^* \, A \, {\bf D}^2 u = {\bf D} \, \big( u^* \, A \, {\bf D} u \big) -({\bf D} u)^* \, A \, {\bf D}^2 u \, .
\end{equation*}
Since $u$ is real valued, the term $u^* \, A \, {\bf D} u$ coincides with $q(u,{\bf D}u)$, where the matrix of 
the quadratic form $q$ is
\begin{equation*}
\dfrac{1}{2} \, \begin{pmatrix}
0 & A \\
-A & 0 \end{pmatrix} \, .
\end{equation*}
If the existence of the decomposition \eqref{ipp2} holds up to some integer $k$, then Lemma \ref{lemLeibniz} 
gives
\begin{equation*}
u^* \, A \, {\bf D}^{k+1} u = {\bf D} \, \big( u^* \, A \, {\bf D}^k u \big) -({\bf D} u)^* \, A \, {\bf D}^k u 
-({\bf D} u)^* \, A \, {\bf D}^{k+1} u \, .
\end{equation*}
We apply the induction assumption for decomposing the term $({\bf D} u)^* \, A \, {\bf D}^{k+1} u$. There are 
two cases for the remaining term $({\bf D} u)^* \, A \, {\bf D}^k u$. Either $k=2$, and this term is already in an 
irreducible form, or $k \ge 3$, and we can apply the induction assumption, which eventually yields the 
decomposition \eqref{ipp2} up to $k+1$.

Uniqueness of the decomposition \eqref{ipp2} relies on more or less the same arguments as those used in the 
proof of Corollary \ref{coroipp1}. More precisely, assuming that two decompositions \eqref{ipp2} exist, we can 
find a quadratic form $q$ on $\R^{N\, k}$, with a corresponding real symmetric matrix $Q$, and a collection of 
real numbers $\beta_{1,k}, \dots, \beta_{k-1,k}$ that satisfy\footnote{Here the vectors $x_0,\dots,x_k$ are real.}
\begin{equation*}
X_1^* \, Q \, X_1 +2\, X_0^* \, Q \, X_1 +\sum_{j=1}^{k-1} \beta_{j,k} \, x_j^* \, A \, x_{j+1} =0 \, ,\quad 
X_0 :=\begin{pmatrix}
x_0 \\
\vdots \\
x_{k-1} \end{pmatrix} ,Ê\, X_1 :=\begin{pmatrix}
x_1 \\
\vdots \\
x_k \end{pmatrix} \, .
\end{equation*}
Identifying the $x_0$ term shows, as in the proof of Corollary \ref{coroipp1}, that the block decomposition of 
$Q$ reads
\begin{equation*}
Q =\begin{pmatrix}
0 & \dots & \dots & 0 \\
\vdots & & & \\
\vdots & & \widetilde{Q} & \\
0 & & & \end{pmatrix} \, ,
\end{equation*}
which means that the following relation holds for all vectors $x_1,\dots,x_k \in \R^N$:
\begin{equation*}
Y_1^* \, \widetilde{Q} \, Y_1 +2\, Y_0^* \, \widetilde{Q} \, Y_1 +\sum_{j=1}^{k-1} \beta_{j,k} \, x_j^* \, A \, x_{j+1} 
=0 \, ,\quad Y_0 :=\begin{pmatrix}
x_1 \\
\vdots \\
x_{k-1} \end{pmatrix} ,Ê\, Y_1 :=\begin{pmatrix}
x_2 \\
\vdots \\
x_k \end{pmatrix} \, .
\end{equation*}
Here the proof differs slightly from that of Corollary \ref{coroipp1} since there is no quadratic term with respect 
to $x_1$. Instead, we identify the quadratic terms with respect to $(x_1,x_2)$, which amounts to taking first a 
partial derivative with respect to $x_2$ and then a partial derivative with respect to $x_1$. This yields
\begin{equation*}
2\, \widetilde{Q}_{1,1} +\beta_{1,k} \, A =0 \, ,
\end{equation*}
where $\widetilde{Q}_{1,1}$ denotes the upper left block of $\widetilde{Q}$ in its block decomposition. Observe 
now that $\widetilde{Q}$ is symmetric, and therefore so is $\widetilde{Q}_{1,1}$, while $A$ is skew-symmetric 
and $\beta_{1,k}$ is real. Hence $\beta_{1,k}$ is zero and uniqueness of the decomposition \eqref{ipp2} follows 
by induction.
\end{proof}

\subsection{Consequences for Cauchy problems}

In this paragraph, we explain some consequences of Corollaries \ref{coroipp1} and \ref{coroipp2} for showing 
stability of finite difference discretizations of Cauchy problems. We consider a numerical discretization with two 
time levels, that is:
\begin{equation}
\label{cauchy}
\begin{cases}
U_j^{n+1} =Q \, U_j^n \, ,& j\in \Z \, ,\quad n\ge 0 \, ,\\
U_j^0 = f_j \, ,& j\in \Z \, ,
\end{cases}
\end{equation}
with $(f_j)_{j \in \Z} \in \ell^2$, and
\begin{equation*}
Q=\sum_{\ell =-r}^p A_\ell \, {\bf T}^\ell \, ,\quad \sum_{\ell =-r}^p A_\ell =I \, .
\end{equation*}
The latter consistency assumption allows us to express the finite difference operator $Q$ as a sum of 
discrete derivatives. Namely, we write
\begin{equation*}
{\bf T}^r \, (Q-I)=\sum_{\underset{\ell \neq 0}{\ell=-r,}}^p A_\ell \, ({\bf T}^{r+\ell} -{\bf T}^r) \, ,
\end{equation*}
and then decompose each ${\bf T}^{r+\ell} -{\bf T}^r$ as a linear combination of ${\bf D},\dots,{\bf D}^{p+r}$ 
(which amounts to decomposing the polynomial $X^{r+\ell}-X^r$ on the family $(X-1,\dots,(X-1)^{p+r})$ which 
forms a basis of the space of real polynomials that vanish at $1$ and whose degree is not larger than $p+r$). 
Summing up, the operator $Q$ can be written equivalently as
\begin{equation}
\label{decompQ}
Q=I +{\bf T}^{-r} \, \sum_{\ell=1}^{p+r} \widetilde{A}_\ell \, {\bf D}^\ell \, ,
\end{equation}
for suitable matrices $\widetilde{A}_1,\dots,\widetilde{A}_{p+r}$ whose expression can be obtained from 
$A_{-r},\dots,A_p$. It is rather clear that all matrices $\widetilde{A}_1,\dots,\widetilde{A}_{p+r}$ are real, 
and they are symmetric if $A_{-r},\dots,A_p$ are symmetric (which we shall not assume, but this might 
simplify some of the calculations below in some given situation).

The following Lemma is a direct consequence of Corollaries \ref{coroipp1} and \ref{coroipp2}.

\begin{lemma}
\label{lemdecomp}
There exists a quadratic form $q$ on $\R^{N\, (p+r)}$, some real symmetric matrices $S_1,\dots,S_{p+r}$ 
and some real skew-symmetric matrices $\widetilde{S}_1,\dots,\widetilde{S}_{p+r-1}$ such that for all 
sequence $U$ with values in $\R^N$, there holds
\begin{align}
2\, U^* \, (Q-I) \, U +|(Q-I) \, U|^2 ={\bf T}^{-r} \, {\bf D} \, \Big( q(U,\dots,{\bf D}^{p+r-1} U) \Big) 
&+{\bf T}^{-r} \, \sum_{\ell=1}^{p+r} ({\bf D}^\ell U)^* \, S_\ell \, {\bf D}^\ell U \notag \\
&+{\bf T}^{-r} \, \sum_{\ell=1}^{p+r-1} ({\bf D}^\ell U)^* \, \widetilde{S}_\ell \, {\bf D}^{\ell+1} U \, .\label{ippQ}
\end{align}
If the sequence $U$ is indexed by $j \ge 1-r$, then \eqref{ippQ} is valid for all indeces $j \ge 1$, while if 
the sequence $U$ is indexed by $\Z$, then \eqref{ippQ} is valid on all $\Z$.

In particular, the solution $(U_j^n)$ to \eqref{cauchy} satisfies
\begin{equation}
\label{energybalance}
\forall \, n \in \N \, ,\quad \sum_{j \in \Z} |U_j^{n+1}|^2 -\sum_{j \in \Z} |U_j^n|^2 
=\sum_{j \in \Z} \sum_{\ell=1}^{p+r} ({\bf D}^\ell U_j^n)^* \, S_\ell \, {\bf D}^\ell U_j^n 
+\sum_{j \in \Z} \sum_{\ell=1}^{p+r-1} ({\bf D}^\ell U_j^n)^* \, \widetilde{S}_\ell \, {\bf D}^{\ell+1} U_j^n \, .
\end{equation}
\end{lemma}

\noindent The decomposition \eqref{ippQ} is unique provided that $Q$ is not the identity operator.

\begin{proof}
The existence of the decomposition \eqref{ippQ} is indeed an easy consequence of Corollaries \ref{coroipp1} 
and \ref{coroipp2}. Due to \eqref{decompQ}, the term $U^* \, (Q-I) \, U$ is a sum of terms of the form
\begin{equation*}
U^* \, ({\bf T}^{-r} \, \widetilde{A}_\ell \, {\bf D}^\ell U) ={\bf T}^{-r} \, \Big( ({\bf T}^r U)^* \, \widetilde{A}_\ell \, 
{\bf D}^\ell U \Big) \, ,
\end{equation*}
which can be written as a (real) linear combination of terms of the form $({\bf D}^{\ell_1} U)^* \, \widetilde{A}_{\ell_2} 
\, {\bf D}^{\ell_2} U$ by simply expanding ${\bf T}^r$ as a linear combination of $I,\dots,{\bf D}^r$ (which is nothing 
but the binomial identity). We then split $\widetilde{A}_{\ell_2}$ as the sum of its symmetric and skew-symmetric parts 
and apply Corollaries \ref{coroipp1} and \ref{coroipp2} (if $\ell_1=\ell_2$, nothing needs to be done). The term 
$|(Q-I) \, U|^2$ can also be written under the form on the right hand-side of \eqref{ippQ} since it is a sum of terms 
of the form
\begin{equation*}
{\bf T}^{-r} \, \Big( \Big( \widetilde{A}_{\ell_1} \, {\bf D}^{\ell_1} U \Big)^* \, \widetilde{A}_{\ell_2} \, {\bf D}^{\ell_2} U 
\Big) ={\bf T}^{-r} \, \Big( ({\bf D}^{\ell_1} U)^* \, \Big( \widetilde{A}_{\ell_1}^* \, \widetilde{A}_{\ell_2} \Big) \, 
{\bf D}^{\ell_2} U \Big) \, ,
\end{equation*}
and it only remains to split $\widetilde{A}_{\ell_1}^* \, \widetilde{A}_{\ell_2}$ as the sum of its symmetric and 
anti-symmetric parts and to apply Corollaries \ref{coroipp1} and \ref{coroipp2} (if $\ell_1=\ell_2$, nothing needs 
to be done).

The energy balance \eqref{energybalance} follows by observing that the sum on $\Z$ of the discrete derivative 
${\bf D} \, q$ vanishes. The remaining terms incorporate the (possible) dissipative behavior of the discretization.
\end{proof}

As a concrete example, let us explain Lemma \ref{lemdecomp} for three points schemes and scalar equations. 
In that case, $N=1$ so that there is no skew-symmetric matrix except $0$, and the scheme reads
\begin{equation*}
U_j^{n+1} =a_{-1} \, U_{j-1}^n +a_0 \, U_j^n +a_1 \, U_{j+1}^n \, ,
\end{equation*}
with a triple of real numbers $a_{-1},a_0,a_1$ that satisfies $a_{-1}+a_0+a_1=1$. In that case, the integration 
by parts procedure leads to the relation
\begin{multline*}
2\, U_j^n \, (Q-I) \, U_j^n +|(Q-I) \, U_j^n|^2 ={\bf T}^{-1} \, {\bf D} \, \Big( (a_1-a_{-1}) \, |U_j^n|^2 
+2\, a_1 \, U_j^n \, {\bf D}\, U_j^n +a_1 \, (a_1-a_{-1}) \, |{\bf D} \, U_j^n|^2 \Big) \\
+{\bf T}^{-1} \, \Big( d_1 \, |{\bf D} \, U_j^n|^2 +d_2 \, |{\bf D}^2 \, U_j^n|^2 \Big) \, ,
\end{multline*}
with
\begin{equation*}
d_1 :=-\dfrac{a_1+a_{-1}}{2} +(a_1-a_{-1})^2 \, ,\quad d_2 :=a_1 \, a_{-1} \, .
\end{equation*}
In that case, stability for the Cauchy problem, that is fulfillment of Assumption \ref{assumption1}, is equivalent 
to the property
\begin{equation*}
\max \, (d_1,d_1 +4\, d_2) \le 0 \, ,
\end{equation*}
or, in other words,
\begin{equation*}
\max \, \Big( (a_1-a_{-1})^2,(a_1+a_{-1})^2 \Big) \le \dfrac{1-a_0}{2} \, .
\end{equation*}
For the upwind, Lax-Friedrichs and Lax-Wendroff schemes, one gets the standard CFL condition 
$\lambda \, |a| \le 1$, with $a$ the velocity of the transport equation one is willing to approximate.

\subsection{Proof of Corollary \ref{coro1}}

We consider the numerical scheme \eqref{numibvp} with $s=0$, zero interior source term and zero boundary source 
term. Writing $Q$ instead of $Q_0$ for simplicity, the scheme reads
\begin{equation}
\label{schema}
\begin{cases}
U_j^{n+1} =Q \, U_j^n \, ,& j\ge 1\, ,\quad n\ge 0 \, ,\\
U_j^{n+1} =B_{j,-1} \, U_1^{n+1} +B_{j,0} \, U_1^n \, ,& j=1-r,\dots,0\, ,\quad n\ge 0 \, ,\\
U_j^0 = f_j \, ,& j\ge 1-r\, ,
\end{cases}
\end{equation}
with $(f_j)_{j \ge 1-r} \in \ell^2$, and
\begin{equation*}
Q=\sum_{\ell =-r}^p A_\ell \, {\bf T}^\ell \, ,\quad \sum_{\ell =-r}^p A_\ell =I \, .
\end{equation*}

We use the decomposition \eqref{ippQ} of $Q$. The solution\footnote{It is assumed here that the initial condition 
consists of real vectors, so that the solution to \eqref{schema} is real. The extension to complex sequences is 
straightforward because the scheme is linear with real coefficients.} $(U_j^n)$ to \eqref{schema} satisfies
\begin{multline*}
\forall \, j \ge 1 \, ,\quad |U_j^{n+1}|^2 -|U_j^n|^2 =2\, (U_j^n)^* \, (Q-I) \, U_j^n +|(Q-I) \, U_j^n|^2 
={\bf T}^{-r} \, {\bf D} \, \Big( q(U_j^n,\dots,{\bf D}^{p+r-1} U_j^n) \Big) \\
+{\bf T}^{-r} \, \sum_{\ell=1}^{p+r} ({\bf D}^\ell U_j^n)^* \, S_\ell \, {\bf D}^\ell U_j^n 
+{\bf T}^{-r} \, \sum_{\ell=1}^{p+r-1} ({\bf D}^\ell U_j^n)^* \, \widetilde{S}_\ell \, {\bf D}^{\ell+1} U_j^n \, .
\end{multline*}
Summing with respect to $j \ge 1$, we get
\begin{align}
\sum_{j \ge 1} |U_j^{n+1}|^2 -\sum_{j \ge 1} |U_j^n|^2 =&-q(U_{1-r}^n,\dots,{\bf D}^{p+r-1} U_{1-r}^n) \notag \\
&+\sum_{j \ge 1-r} \sum_{\ell=1}^{p+r} ({\bf D}^\ell U_j^n)^* \, S_\ell \, {\bf D}^\ell U_j^n 
+\sum_{j \ge 1-r} \sum_{\ell=1}^{p+r-1} ({\bf D}^\ell U_j^n)^* \, \widetilde{S}_\ell \, {\bf D}^{\ell+1} U_j^n \, ,\label{energy1}
\end{align}
where, comparing with Lemma \ref{lemdecomp}, the novelty is the "boundary" term 
$q(U_{1-r}^n,\dots,{\bf D}^{p+r-1} U_{1-r}^n)$.

Our goal now is to estimate the terms which appear in the second line of \eqref{energy1}. Following an argument 
already used in \cite{wu,jfcag}, we extend the sequence $(U_j^n)_{j \ge 1-r}$ by zero for $j \le -r$, and still denote it 
$(U_j^n)$. This extended sequence belongs to $\ell^2(\Z)$, and we can therefore use the assumption of Corollary 
\ref{coro1} on the action of $Q$ on $\ell^2(\Z)$. We obtain
\begin{equation*}
\sum_{j \in \Z} 2\, (U_j^n)^* \, (Q-I) \, U_j^n +|(Q-I) \, U_j^n|^2 \le 0 \, ,
\end{equation*}
which, using the decomposition \eqref{decompQ} and the fact that $U_j^n$ vanishes for $j \le -r$, gives
\begin{multline}
\label{energy2}
\sum_{j \ge 1-r} \sum_{\ell=1}^{p+r} ({\bf D}^\ell U_j^n)^* \, S_\ell \, {\bf D}^\ell U_j^n 
+\sum_{j \ge 1-r} \sum_{\ell=1}^{p+r-1} ({\bf D}^\ell U_j^n)^* \, \widetilde{S}_\ell \, {\bf D}^{\ell+1} U_j^n \\
\le -\sum_{j=1-p-2r}^{-r} \sum_{\ell=1}^{p+r} ({\bf D}^\ell U_j^n)^* \, S_\ell \, {\bf D}^\ell U_j^n 
-\sum_{j=1-p-2r}^{-r} \sum_{\ell=1}^{p+r-1} ({\bf D}^\ell U_j^n)^* \, \widetilde{S}_\ell \, {\bf D}^{\ell+1} U_j^n \, .
\end{multline}
The combination of \eqref{energy1} and \eqref{energy2} shows that there exists a quadratic form $q_\flat$ on 
$\R^{N\, (p+r)}$, which only depends on $Q$, such that any solution to \eqref{schema} satisfies
\begin{equation*}
\sum_{j \ge 1} |U_j^{n+1}|^2 -\sum_{j \ge 1} |U_j^n|^2 \le q_\flat (U_{1-r}^n,\dots,U_p^n) \, .
\end{equation*}
In particular, there exists a numerical constant $C$, that only depends on the operator $Q$ and not on the solution 
$(U_j^n)$ to \eqref{schema}, such that
\begin{equation*}
\sum_{j \ge 1} |U_j^{n+1}|^2 -\sum_{j \ge 1} |U_j^n|^2 \le C \, \sum_{j=1-r}^p |U_j^n|^2 \, .
\end{equation*}
Summing with respect to $n$, and using the fact that $\Delta t/\Delta x$ is constant, we end up with
\begin{equation}
\label{energy3}
\sup_{n \in \N} \, \sum_{j \ge 1} \Delta x \, |U_j^n|^2 \le \sum_{j \ge 1} \Delta x \, |f_j|^2 
+C \, \sum_{n \ge 0} \sum_{j=1-r}^p \Delta t \, |U_j^n|^2 \, .
\end{equation}
Let us now observe that for all $n \in \N$, we have
\begin{equation*}
\sum_{j=1-r}^0 \Delta x \, |U_j^n|^2 \le \dfrac{1}{\lambda} \, \sum_{j=1-r}^0 \Delta t \, |U_j^n|^2  
\le \dfrac{1}{\lambda} \, \sum_{\nu \in \N} \sum_{j=1-r}^0 \Delta t \, |U_j^\nu|^2  \, ,
\end{equation*}
so that the left hand-side of \eqref{energy3} can be slightly increased in order to obtain
\begin{equation*}
\sup_{n \in \N} \, \sum_{j \ge 1-r} \Delta x \, |U_j^n|^2 \le \sum_{j \ge 1} \Delta x \, |f_j|^2 
+C \, \sum_{n \ge 0} \sum_{j=1-r}^p \Delta t \, |U_j^n|^2 \, .
\end{equation*}
We then use Theorem \ref{thm1} to control the trace of $(U_j^n)$ in terms of the initial condition $(f_j)$, 
which is done by letting $\gamma$ tend to zero in \eqref{estimthm}, and this completes the proof of Corollary 
\ref{coro1}.

\begin{remark}
The above derivation of the semigroup estimate for the solution $(U_j^n)$ to \eqref{schema} heavily relies on 
the assumption $\| Q_0 \|_{\ell^2(\Z) \rightarrow \ell^2(\Z)}=1$, which in view of the consistency assumption on 
the $A_\ell$'s, is equivalent to $\| Q_0 \|_{\ell^2(\Z) \rightarrow \ell^2(\Z)} \le 1$. This property is called "strong 
stability" in \cite{strang}, see also \cite{tadmor2}, though "strong stability" in this context should not be mixed 
up with Definition \ref{defstab1}.

The exact same assumption on $Q_0$ is the cornerstone of the analysis in \cite{jfcag}. Since \cite{wu} deals 
with scalar problems, this assumption is also present, though hidden, in \cite{wu}. However, the method we 
use here is completely different from the one in \cite{wu,jfcag} and bypasses the introduction of Dirichlet or 
other auxiliary boundary conditions. Unlike \cite{wu,jfcag}, we use here the consistency of the discretized 
hyperbolic operator in order to derive an "integration by parts formula", which connects the time derivative 
of the $\ell^2$ norm of $(U_j^n)$ with the trace of $(U_j^n)$ on the first space meshes.

We aim in a near future to extend the derivation of such an "integration by parts formula" to numerical schemes 
with arbitrarily many time levels, which would imply, with the help of Theorem \ref{thm1}, a semigroup estimate 
for the solution to \eqref{numibvp} and therefore a positive answer to the uniform power boundedness conjecture.
\end{remark}

\appendix
\section{On the non-glancing condition}
\label{appA}

The goal of this Appendix is to show that the validity of Proposition \ref{prop1} is equivalent to the 
non-occurrence of glancing wave packets. This uses similar constructions as those in \cite{trefethen3}, 
namely we use {\it discrete geometric optics} expansions. Opposite to \cite{trefethen3}, we use here a 
{\it fully discrete} framework, namely we only deal with piecewise constant functions. This has a major 
impact on the arguments we use. While in \cite[Lemma 5.1]{trefethen3}, $L^\infty$ error bounds are derived 
by using decay of the Fourier transform (or, equivalently, smoothness of the functions), the framework of 
step functions yields Fourier transforms that have no better than $L^2$ decay (and certainly not $L^1$). 
Hence the derivation of $L^\infty$ error bounds is more intricate than in \cite{trefethen3}, and we 
pay special attention to the rigorous justification of our error bound below. Our result is the following.

\begin{proposition}
\label{propA}
Let Assumptions \ref{assumption1} and \ref{assumption2} be satisfied. Assume furthermore that there exists 
a constant $C>0$ such that for all $\Delta t \in \, ]0,1]$, and for all solution to the fully discrete Cauchy problem
\begin{equation}
\label{cauchys}
\begin{cases}
V_j^{n+1} = {\dps \sum_{\sigma=0}^s} Q_\sigma \, V_j^{n-\sigma} \, ,& j\in \Z \, ,\quad n\ge s \, ,\\
V_j^n = f_j^n \, ,& j \in \Z\, ,\quad n=0,\dots,s \, ,
\end{cases}
\end{equation}
there holds
\begin{equation}
\label{estimtrace}
\sum_{n \ge 0} \Delta t \, |V_0^n|^2 \le C \, \sum_{n=0}^s \, \sum_{j \in \Z} \Delta x \, |f_j^n|^2 \, .
\end{equation}
Then Assumption \ref{assumption3} is satisfied.
\end{proposition}

The proof of Proposition \ref{propA} is based on high frequency asymptotics for solutions to \eqref{cauchys}. 
We first state independently a Lemma which gives the expression of the Fourier transform of a piecewise 
constant "highly oscillating" function\footnote{Of course the maximal frequency that is compatible with the 
mesh is $2\, \pi/\Delta x$ so high frequency in our discrete setting means a frequency of order $1/\Delta x$.}.

\begin{lemma}
\label{lemA}
Let $a$ denote a Schwartz function from $\R$ to $\C^q$ for some $q \in \N$. Given $\underline{\xi} \in \R$ 
and $\Delta x>0$, we consider the step function
\begin{equation*}
\forall \, j \in \Z \, ,\quad \forall \, x \in [j\, \Delta x,(j+1) \, \Delta x[ \, ,\quad 
a_\Delta (x) := {\rm e}^{i\, j \, \underline{\xi}} \, \, a(j\, \Delta x) \, .
\end{equation*}
Then $a_\Delta \in L^1(\R) \cap L^2(\R)$ and its Fourier transform is given by
\begin{equation*}
\forall \, \xi \in \R \, ,\quad \widehat{a_\Delta}(\xi)=\dfrac{1-{\rm e}^{-i\, \Delta x \, \xi}}{i\, \Delta x \, \xi} \, 
\sum_{m \in \Z} \widehat{a} \left( \xi -\dfrac{\underline{\xi} +2\, m\, \pi}{\Delta x} \right) \, .
\end{equation*}
\end{lemma}

\noindent Observe that the function $a_\Delta$ in Lemma \ref{lemA} is a piecewise constant version of the 
"continuous" function
\begin{equation*}
x \in \R \longmapsto {\rm e}^{i \, x \, \underline{\xi}/\Delta x} \, a(x) \, ,
\end{equation*}
which represents a fast oscillation at frequency $\underline{\xi}/\Delta x$ ($\Delta x$ is meant to be small 
while $\underline{\xi}$ is fixed), with a slowly varying smooth envelope $a$.

\begin{proof}[Proof of Lemma \ref{lemA}]
Due to the fast decay of $a$ at infinity, the Fourier transform of $a_\Delta$ is given by
\begin{equation*}
\widehat{a_\Delta}(\xi)= \sum_{j \in \Z} {\rm e}^{i\, j \, \underline{\xi}} \, a(j\, \Delta x) \, 
\int_{j\, \Delta x}^{(j+1) \, \Delta x} {\rm e}^{-i\, x \, \xi} \, {\rm d}x 
=\dfrac{1-{\rm e}^{-i\, \Delta x \, \xi}}{i\, \xi} \sum_{j \in \Z} {\rm e}^{-i\, j \, \Delta x \, (\xi-\underline{\xi}/\Delta x)} 
\, a(j\, \Delta x) \, ,
\end{equation*}
and it only remains to apply the so-called Poisson summation formula to obtain the result of Lemma \ref{lemA}.
\end{proof}

\begin{proof}[Proof of Proposition \ref{propA}]
Let us first rewrite \eqref{cauchys} as a scheme with two time levels for an augmented vector. Namely, we 
introduce $W_j^n :=(V_j^{n+s},\dots,V_j^n) \in \C^{N\, (s+1)}$, and rewrite \eqref{cauchys} as
\begin{equation}
\label{cauchys'}
W_j^{n+1} = {\mathcal Q} \, W_j^n \, ,\quad j\in \Z \, ,\quad n\ge 0 \, ,
\end{equation}
with the operator
\begin{equation*}
{\mathcal Q} :=\begin{pmatrix}
Q_0 & \dots & \dots & Q_s \\
I & 0 & \dots & 0 \\
0 & \ddots & \ddots & \vdots \\
0 & 0 & I & 0 \end{pmatrix} \, .
\end{equation*}
The estimate \eqref{estimtrace} can be equivalently rewritten for solutions to \eqref{cauchys'} as
\begin{equation}
\label{estimtrace'}
\sum_{n \ge 0} \Delta t \, |W_0^n|^2 \le C \, \sum_{j \in \Z} \Delta x \, |W_j^0|^2 \, .
\end{equation}

Let us consider some fixed parameter $\underline{\xi} \in [0,2\, \pi[$, a Schwartz function $a$ from $\R$ to 
$\C^{N\, (s+1)}$ and the initial sequence for \eqref{cauchys'}:
\begin{equation*}
\forall \, j \in \Z \, ,\quad W_j^0 :={\rm e}^{i\, j \, \underline{\xi}} \, \, a(j\, \Delta x) \, .
\end{equation*}
For all $n \in \N$, the step function corresponding to the sequence $(W_j^n)_{j \in \Z}$ is denoted 
$W_\Delta^n$. Applying the Fourier transform to \eqref{cauchys'} and using Lemma \ref{lemA}, we have
\begin{equation*}
\widehat{W_\Delta^n} (\xi) =\dfrac{1-{\rm e}^{-i\, \Delta x \, \xi}}{i\, \Delta x \, \xi} \, \sum_{m \in \Z} 
{\mathcal A} \big( {\rm e}^{i\, \Delta x \, \xi} \big)^n \, 
\widehat{a} \left( \xi -\dfrac{\underline{\xi} +2\, m\, \pi}{\Delta x} \right) \, .
\end{equation*}

We now use Assumptions \ref{assumption1} and \ref{assumption2} in order to give a detailed expansion of 
the amplification matrix ${\mathcal A}$ close to $\exp(i\, \underline{\xi})$: there exists an integer $P$ such 
that, on the disk $\{ \eta \in \C \, / \, |\eta -\underline{\xi}|<\delta_0 \}$, ${\mathcal A}$ admits the spectral 
decomposition
\begin{equation*}
{\mathcal A}({\rm e}^{i\, \eta}) =\sum_{p=1}^P {\rm e}^{i\, \omega_p(\eta)} \, \Pi_p(\eta) 
+{\mathcal A}_\sharp (\eta) \, \Pi_\sharp(\eta) \, ,
\end{equation*}
with (scalar) functions $\omega_1,\dots,\omega_P$, rank one projectors $\Pi_1,\dots,\Pi_P$, a rank $N\, (s+1)-P$ 
projector $\Pi_\sharp$ and a square matrix ${\mathcal A}_\sharp$ that has spectral radius less than $1$ for 
all $\eta$. In the latter decomposition, all functions depend holomorphically on $\eta$. The functions $\omega_1, 
\dots,\omega_P$ satisfy
\begin{equation*}
\forall \, p=1,\dots,P\, ,\quad \omega_p (\underline{\xi}) \in \R \, ,
\end{equation*}
so the $\exp (i\, \omega_p(\eta))$ correspond to the eigenvalues of the amplification matrix that are close to the unit 
circle as $\exp (i\, \eta)$ is close to $\exp (i\, \underline{\xi})$. Of course, we can extend all functions to the disks 
$\{ \eta \in \C \, / \, |\eta -(\underline{\xi}+2\, m \, \pi)|<\delta_0 \}$, $m \in \Z$, by $2\, \pi$-periodicity because 
${\mathcal A}(\exp (i\, \cdot))$ is $2\, \pi$-periodic. The latter spectral decomposition of ${\mathcal A}$ only 
holds "microlocally", that is, locally near $\underline{\xi} +2\, \pi \, \Z$. To avoid technicalities, we assume that 
$a$ satisfies
\begin{equation*}
a \in {\mathcal C}^\infty_0(\R) \, ,\quad \text{\rm Supp } \widehat{a} \subset [-\delta_0/2,\delta_0/2] \, .
\end{equation*}
In this way, the expression of $\widehat{W_\Delta^n}$ splits into
\begin{multline}
\label{fourierWn}
\widehat{W_\Delta^n} (\xi) =\dfrac{1-{\rm e}^{-i\, \Delta x \, \xi}}{i\, \Delta x \, \xi} \, \sum_{m \in \Z} \, 
\sum_{p=1}^P {\rm e}^{i\, n \, \omega_p(\xi \, \Delta x)} \, 
\Pi_p(\xi \, \Delta x) \, \widehat{a} \left( \xi -\dfrac{\underline{\xi} +2\, m\, \pi}{\Delta x} \right) \\
+\dfrac{1-{\rm e}^{-i\, \Delta x \, \xi}}{i\, \Delta x \, \xi} \, \sum_{m \in \Z} {\mathcal A}_\sharp (\xi \, \Delta x)^n \, 
\Pi_\sharp(\xi \, \Delta x) \, \widehat{a} \left( \xi -\dfrac{\underline{\xi} +2\, m\, \pi}{\Delta x} \right) \, .
\end{multline}

Following \cite{trefethen1,trefethen3}, we define the group velocities ${\bf v}_p :=-\omega_p'(\underline{\xi})/\lambda$, 
which by Assumption \ref{assumption1}, are known to be real (see, for instance, \cite[Lemma 3.2]{trefethen3}). In the 
notation of Assumption \ref{assumption2}, the group velocity is equivalently given by ${\bf v}_p=-\underline{\kappa} \, 
\zeta_p'(\underline{\kappa})/(\lambda \, \underline{z})$. In particular, Assumption \ref{assumption3} is valid provided 
that the scheme does not admit any wave packet with a vanishing group velocity. We introduce the "approximate" 
solution to \eqref{cauchys'} by defining:
\begin{equation*}
\forall \, (j,n) \in \Z \times \N \, ,\quad {\mathcal W}_j^n :=\sum_{p=1}^P 
{\rm e}^{i\, (n\, \omega_p(\underline{\xi})+j \, \underline{\xi})} \, \Pi_p(\underline{\xi}) \, 
a(j\, \Delta x -n \, \Delta t \, {\bf v}_p) \, ,
\end{equation*}
which represents a sum of highly oscillating signals with phase velocity $-\omega_p(\underline{\xi})/(\lambda \, 
\underline{\xi})$, and corresponding smooth envelopes that propagate at the group velocity ${\bf v}_p$. According 
to Lemma \ref{lemA}, the Fourier transform of the corresponding piecewise constant function is given by
\begin{equation}
\label{fouriercalWn}
\widehat{{\mathcal W}_\Delta^n}(\xi) =\dfrac{1-{\rm e}^{-i\, \Delta x \, \xi}}{i\, \Delta x \, \xi} \, \sum_{m \in \Z} 
\, \sum_{p=1}^P {\rm e}^{i\, n\, \omega_p(\underline{\xi}) +i \, n \, \Delta x \, \omega_p'(\underline{\xi})\, 
(\xi -(\underline{\xi} +2\, m\, \pi)/\Delta x)} \, \Pi_p(\underline{\xi}) \, 
\widehat{a} \left( \xi -\dfrac{\underline{\xi} +2\, m\, \pi}{\Delta x} \right) \, .
\end{equation}
We are now going to estimate the error $W_0^n -{\mathcal W}_0^n$.

Let us define the error:
\begin{equation*}
\forall \, (j,n) \in \Z \times \N \, ,\quad e_j^n :=W_j^n -{\mathcal W}_j^n \, .
\end{equation*}
The expressions \eqref{fourierWn} and \eqref{fouriercalWn} show that the Fourier transform $\widehat{e_\Delta^n}$ 
splits as:
\begin{equation*}
\widehat{e_\Delta^n} =\sum_{p=1}^P \varepsilon^n_{1,p}(\xi) +\varepsilon^n_{2,p}(\xi) 
+\varepsilon^n_\sharp(\xi) \, ,
\end{equation*}
with, for all $p=1,\dots,P$,
\begin{align}
\varepsilon^n_{1,p}(\xi) &:=\dfrac{1-{\rm e}^{-i\, \Delta x \, \xi}}{i\, \Delta x \, \xi} \, \sum_{m \in \Z} 
\left( {\rm e}^{i\, n \, \omega_p(\xi \, \Delta x)} -{\rm e}^{i\, n\, \omega_p(\underline{\xi}) 
+i \, n \, \omega_p'(\underline{\xi})\, (\xi\, \Delta x -\underline{\xi} -2\, m\, \pi)} \right) \, 
\Pi_p(\underline{\xi}) \, \widehat{a} \left( \xi -\dfrac{\underline{\xi} +2\, m\, \pi}{\Delta x} \right) \, ,\label{defepsilon1}\\
\varepsilon^n_{2,p}(\xi) &:=\dfrac{1-{\rm e}^{-i\, \Delta x \, \xi}}{i\, \Delta x \, \xi} \, \sum_{m \in \Z} 
{\rm e}^{i\, n \, \omega_p(\xi \, \Delta x)} \, \big( \Pi_p(\xi \, \Delta x) -\Pi_p(\underline{\xi} +2\, m\, \pi) \big) 
\, \widehat{a} \left( \xi -\dfrac{\underline{\xi} +2\, m\, \pi}{\Delta x} \right) \, ,\label{defepsilon2}
\end{align}
and
\begin{equation}
\label{defepsilondiese}
\varepsilon^n_\sharp(\xi) :=\dfrac{1-{\rm e}^{-i\, \Delta x \, \xi}}{i\, \Delta x \, \xi} \, \sum_{m \in \Z} 
{\mathcal A}_\sharp (\xi \, \Delta x)^n \, \Pi_\sharp(\xi \, \Delta x) \, 
\widehat{a} \left( \xi -\dfrac{\underline{\xi} +2\, m\, \pi}{\Delta x} \right) \, .
\end{equation}

Let us first estimate the $L^2$ norm of $\varepsilon^n_{1,p}$ in \eqref{defepsilon1}. We fix a time $T>0$, 
and consider integers $n$ such that $n \, \Delta t \le T$. Since $\omega_p(\underline{\xi})$ and $\omega_p' 
(\underline{\xi})$ are real, there holds
\begin{multline*}
|\varepsilon^n_{1,p}(\xi)| \le C\, \dfrac{|1-{\rm e}^{-i\, \Delta x \, \xi}|}{\Delta x \, |\xi|} \, \sum_{m \in \Z} 
\Big| {\rm e}^{i\, n \, \omega_p(\xi \, \Delta x)-i\, n\, \omega_p(\underline{\xi}+2\, m \, \pi) 
-i \, n \, \omega_p'(\underline{\xi}+2\, m \, \pi)\, (\xi\, \Delta x -\underline{\xi} -2\, m\, \pi)} -1 \Big| \\
\left| \widehat{a} \left( \xi -\dfrac{\underline{\xi} +2\, m\, \pi}{\Delta x} \right) \right| \, .
\end{multline*}
There is no loss of generality in assuming $\delta_0/\lambda <\pi$. Then the support property of $\widehat{a}$ 
shows that in the latter sum with respect to $m \in \Z$, at most one term is nonzero. Consequently, there holds
\begin{multline*}
|\varepsilon^n_{1,p}(\xi)|^2 \le C \, \dfrac{|1-{\rm e}^{-i\, \Delta x \, \xi}|^2}{\Delta x^2 \, \xi^2} \, 
\sum_{m \in \Z} \Big| {\rm e}^{i\, n \, \omega_p(\xi \, \Delta x)-i\, n\, \omega_p(\underline{\xi}+2\, m \, \pi) 
-i \, n \, \omega_p'(\underline{\xi}+2\, m \, \pi)\, (\xi\, \Delta x -\underline{\xi} -2\, m\, \pi)} -1 \Big|^2 \\
\left| \widehat{a} \left( \xi -\dfrac{\underline{\xi} +2\, m\, \pi}{\Delta x} \right) \right|^2 \, ,
\end{multline*}
and because of the limitation $n \, \Delta t \le T$, there holds\footnote{Here we use Assumption \ref{assumption1} 
to obtain that the imaginary part of $\omega_p(\xi \, \Delta x)$ is nonpositive.}
\begin{equation*}
\Big| {\rm e}^{i\, n \, \omega_p(\xi \, \Delta x)-i\, n\, \omega_p(\underline{\xi}+2\, m \, \pi) 
-i \, n \, \omega_p'(\underline{\xi}+2\, m \, \pi)\, (\xi\, \Delta x -\underline{\xi} -2\, m\, \pi)} -1 \Big| 
\le C\, T \, \Delta x \, \left( \xi -\dfrac{\underline{\xi} +2\, m\, \pi}{\Delta x} \right)^2 
\le C\, T \, \Delta x \, ,
\end{equation*}
on the support of $\widehat{a} (\xi-(\underline{\xi} +2\, m\, \pi)/\Delta x)$.  We thus derive the bound
\begin{align*}
\int_\R |\varepsilon^n_{1,p}(\xi)|^2 \, {\rm d}\xi &\le C \, T^2 \, \Delta x^2 \, \sum_{m \in \Z} \int_\R 
\dfrac{|1-{\rm e}^{-i\, \Delta x \, \xi}|^2}{\Delta x^2 \, \xi^2} \, 
\left| \widehat{a} \left( \xi -\dfrac{\underline{\xi} +2\, m\, \pi}{\Delta x} \right) \right|^2 \, {\rm d}\xi \\
&\le C \, \| \widehat{a} \|_{L^\infty}^2 \, T^2 \, \Delta x^2 \, \sum_{m \in \Z} 
\int_{(\underline{\xi} +2\, m\, \pi)/\Delta x-\delta_0/2}^{(\underline{\xi} +2\, m\, \pi)/\Delta x+\delta_0/2} 
\dfrac{|1-{\rm e}^{-i\, \Delta x \, \xi}|^2}{\Delta x^2 \, \xi^2} \, {\rm d}\xi \\
&\le C \, T^2 \, \Delta x \, \sum_{m \in \Z} 
\int_{\underline{\xi} +2\, m\, \pi-\delta_0 \, \Delta x/2}^{\underline{\xi} +2\, m\, \pi +\delta_0 \, \Delta x/2} 
\dfrac{|1-{\rm e}^{-i\, \eta}|^2}{\eta^2} \, {\rm d}\eta \\
&\le C \, T^2 \, \Delta x \, \sum_{m \in \Z} 
\int_{\underline{\xi} +2\, m\, \pi-\delta_0 \, \Delta x/2}^{\underline{\xi} +2\, m\, \pi +\delta_0 \, \Delta x/2} 
\dfrac{1}{1+\eta^2} \, {\rm d}\eta \le C \, T^2 \, \Delta x^2 \, ,
\end{align*}
with a constant $C>0$ that is uniform with respect to $T>0$ and $\Delta t \in \, ]0,1]$. (Recall that the ratio 
$\Delta t/\Delta x$ is kept fixed.) Similarly, the error $\varepsilon^n_{2,p}$ in \eqref{defepsilon2} 
satisfies\footnote{Here we use again Assumption \ref{assumption1} in order to have 
$|\exp (i\, n \, \omega_p(\xi \, \Delta x))| \le 1$ uniformly in $n$.}
\begin{equation*}
\int_\R |\varepsilon^n_{2,p}(\xi)|^2 \, {\rm d}\xi \le C\, \Delta x^2 \, ,
\end{equation*}
with a constant $C>0$ that is uniform with respect to $T>0$ and $\Delta t \in \, ]0,1]$.

If ${\mathcal W}_j^n$ is meant to be a good approximation of $W_j^n$, including for small values of $n$, 
then the term $\varepsilon^n_\sharp$ in \eqref{defepsilondiese} is meant to be small. In order to achieve 
this, we assume that $a$ satisfies the polarization condition
\begin{equation*}
\Pi_\sharp (\underline{\xi}) \, a =0 \, .
\end{equation*}
Let us now derive an $L^2$ bound for $\varepsilon^n_\sharp$. Shrinking $\delta_0$ is necessary, there 
is no loss of generality in assuming that the matrix ${\mathcal A}_\sharp$ in the spectral decomposition 
of ${\mathcal A}$ is power bounded:
\begin{equation*}
\sup_{n \in \N} \, |{\mathcal A}_\sharp(\eta)^n| \le C \, ,
\end{equation*}
with a constant $C>0$ that is uniform with respect to $\eta$ as long as $|\eta -(\underline{\xi}+2\, m \, \pi)| 
\le \delta_0/2$. (We shall not even use here the exponential decay in time of the $\sharp$ component.) 
Performing the same kind of analysis as for the terms $\varepsilon_{2,p}^n$, the error $\varepsilon^n_\sharp$ 
in \eqref{defepsilondiese} satisfies
\begin{equation*}
\int_\R |\varepsilon^n_\sharp(\xi)|^2 \, {\rm d}\xi \le C\, \Delta x^2 \, ,
\end{equation*}
with a constant $C>0$ that is uniform with respect to $T>0$ and $\Delta t \in \, ]0,1]$.

By Plancherel Theorem, we have proved the bound
\begin{equation*}
\sum_{j \in \Z} \Delta x \, |e_j^n|^2 \le C\, \Delta x^2 \, (1+T^2) \, ,
\end{equation*}
for all $n$ such that $n \, \Delta t \le T$ and a constant $C$ that is uniform with respect to $T>0$ and 
$\Delta x \in \, ]0,1]$. In particular, there holds
\begin{equation}
\label{error}
\| W^n_\Delta -{\mathcal W}^n_\Delta \|_{L^\infty(\R)}^2 
=\sup_{j \in \Z} |W_j^n -{\mathcal W}_j^n|^2 \le C\, \Delta x \, (1+T^2) \, ,
\end{equation}
which gives an $L^\infty$ bound for the error between the exact and approximate solutions provided that 
$\widehat{a}$ has sufficiently narrow support, and $a$ is suitably polarized ($\Pi_\sharp (\underline{\xi}) \, a =0$).

The proof of Proposition \ref{propA} is now almost complete. Indeed, let us assume that Assumption 
\ref{assumption3} is not valid. Up to reordering, this means that for some $\underline{\xi}$, the group 
velocity ${\bf v}_1$ is zero. We use the previous construction of high frequency solutions to \eqref{cauchys'}. 
Choosing $a$ such that the (more restrictive) polarization condition $\Pi_1(\underline{\xi}) \, a=a$ holds, 
the expression of the approximate solution ${\mathcal W}$ reduces to
\begin{equation*}
\forall \, (j,n) \in \Z \times \N \, ,\quad {\mathcal W}_j^n := 
{\rm e}^{i\, (n\, \omega_1(\underline{\xi})+j \, \underline{\xi})} \, a(j\, \Delta x) \, .
\end{equation*}
Let us consider some time $T>0$. The trace estimate \eqref{estimtrace'} gives
\begin{align*}
\sum_{0 \le n \le T/\Delta t} \Delta t \, |{\mathcal W}_0^n|^2 
&\le 2 \, \sum_{0 \le n \le T/\Delta t} \Delta t \, |W_0^n|^2 
+2\, \sum_{1 \le n \le T/\Delta t} \Delta t \, |{\mathcal W}_0^n-W_0^n|^2 \\
&\le C\, \sum_{j \in \Z} \Delta x \, |W_j^0|^2 +C \, \Delta x \, (1+T^2) \, T \, .
\end{align*}
By the smoothness of $a$, there holds
\begin{align*}
\sum_{j \in \Z} \Delta x \, |W_j^0|^2 =\sum_{j \in \Z} \Delta x \, |a(j\, \Delta x)|^2 
&\le C\, \sum_{j \in \Z} \| a \|_{L^2([j\, \Delta x,(j+1)\, \Delta x[)}^2 
+\Delta x^2 \, \| a' \|_{L^2(L^2([j\, \Delta x,(j+1)\, \Delta x[)}^2 \\
&\le C\, \| a \|_{H^1(\R)}^2 \, ,
\end{align*}
uniformly with respect to $\Delta t \in \, ]0,1]$. Summing up, we have shown that, for a suitable constant 
$C>0$ that is uniform with respect to $T>0$ and $\Delta t \in \, ]0,1]$, there holds
\begin{equation*}
(N_T+1) \, \Delta t \, |a(0)|^2 \le C +C\, \Delta x \, (1+T^2) \, T\, ,
\end{equation*}
with $N_T$ the largest integer such that $N_T \, \Delta t \le T$. By first passing to the limit $\Delta t \rightarrow 
0$, we get
\begin{equation*}
T\, |a(0)|^2 \le C \, ,
\end{equation*}
and by passing to the limit $T \rightarrow +\infty$, we get $a(0)=0$, which is obviously a contradiction 
because one can construct the function $a$ that meets all previous requirements (support of $\widehat{a}$, 
smoothness and polarization), together with $a(0) \neq 0$.
\end{proof}

\begin{remark}
The above argument is actually simpler in the PDE context because an accurate description of high frequency 
asymptotics (including $L^\infty$ error bounds) is available for hyperbolic systems, say with constant 
multiplicity. Consider for instance the Cauchy problem
\begin{equation*}
\partial_t u +\sum_{j=1}^d A_j \, \partial_{x_j} u =0 \, ,
\end{equation*}
with a hyperbolic operator of constant multiplicity, that is:
\begin{equation*}
\forall \, \xi=(\xi_1,\dots,\xi_d) \in \R^d \setminus \{ 0 \} \, ,\quad
\det \Big[ \tau \, I+\sum_{j=1}^d \xi_j \, A_j \Big] =\prod_{k=1}^q \big( \tau+\lambda_k(\xi) \big)^{\nu_k} \, ,
\end{equation*}
with (real valued) real analytic semi-simple eigenvalues $\lambda_1,\dots,\lambda_q$. The validity of the 
trace estimate
\begin{equation*}
\int_{\R^+} \int_{\R^{d-1}} |u(t,y,0)|^2 \, {\rm d}y \, {\rm d}t \le C \, \| u(0,\cdot) \|_{L^2(\R^d)}^2 \, ,
\end{equation*}
is equivalent to the fact that there is no glancing wave packet, namely:
\begin{equation*}
\forall \, \xi \neq 0 \, ,\quad \forall \, k=1,\dots,q \, ,\quad \dfrac{\partial \lambda_k(\xi)}{\partial \xi_d} \neq 0 \, .
\end{equation*}
The latter condition is basically never satisfied in dimension $d \ge 2$, and this is one reason why the 
derivation of semigroup estimates in \cite{kajitani,rauch} and followers is so involved.
\end{remark}

\paragraph{Acknowledgments} The discussion in Appendix \ref{appA} originates from a discussion in Les Houches 
with Guy M\'etivier, whom I warmly thank for pointing out several "well-known" results. I also thank Mark Williams 
for providing and clarifying parts of \cite{MT}. Eventually, it is a pleasure to thank Denis Serre, whose interest 
and encouragements over the years have been very stimulating.

\bibliographystyle{alpha}
\bibliography{Semigroupe}
\end{document}